\documentclass{amsart}

\usepackage{amsmath}
\usepackage{amssymb}
\usepackage{amsfonts}
\usepackage{amsthm}
\usepackage{amsopn}
\usepackage{enumerate}
\usepackage[all]{xy}
\usepackage{array}
\usepackage{mathrsfs}
\usepackage{tikz}
\usepackage{tikz-cd}
\usepackage{dsfont}
\usepackage{caption}
\usepackage{hyperref}
\usepackage{comment}
\usepackage[normalem]{ulem} 
\hypersetup{
 colorlinks  = true,  
 urlcolor   = blue,  
 linkcolor  = blue,  
 citecolor  = black   
}
 \definecolor{darkgreen}{RGB}{55,138,0}
 
\mathchardef\mhyphen="2D
\newcommand{\m}{\alpha}
\newcommand{\Z}{\operatorname{Zer}(f)}

\newcommand{\kk}{\mathds{k}}
\newcommand{\NN}{\ensuremath{\mathbb{N}}}

\newcommand{\ZZ}{\ensuremath{\mathbb{Z}}}

\newcommand{\mcS}{\ensuremath{\mathcal{S}}}
\newcommand{\mcP}{\ensuremath{\mathcal{P}}}

\newcommand{\pp}{\mathfrak{p}}

\newcommand{\sfR}{\ensuremath{\mathsf{R}}}
\newcommand{\sfD}{\ensuremath{\mathsf{D}}}

\newcommand{\ulI}{\ensuremath{\mathbf{I}}}
\newcommand{\ulJ}{\ensuremath{\mathbf{J}}}
\newcommand{\ule}{\ensuremath{\mathbf{e}}}
\newcommand{\ulK}{\ensuremath{\mathbf{K}}}

\newcommand{\ra}{\ensuremath{\longrightarrow}}
\newcommand{\sra}{\ensuremath{\rightarrow}}

\newcommand{\Ext}{\operatorname{Ext}}

\newcommand{\Hom}{\operatorname{Hom}}
\newcommand{\uHom}{\operatorname{\underline{Hom}}}

\newcommand{\gldim}{\operatorname{gl dim}}

\newcommand{\st}{\ensuremath{\mathrel{}\middle|\mathrel{}}}
\renewcommand{\th}{\mathrm{th}}
\newcommand\inv{^{-1}}
\newcommand{\Id}{\operatorname{Id}}
\newcommand{\im}{\operatorname{im}}

\def\lcm{\mathop{\operatorname{lcm}}}

\newcommand{\cO}{\mathcal{O}}
\newcommand{\ma}{\ensuremath{^{-}_{\m }}}
\newcommand{\pa}{\ensuremath{^{+}_{\m }}}
\newcommand{\pma}{\ensuremath{^{\pm}_{\m }}}
\newcommand{\mpa}{\ensuremath{^{\mp}_{\m }}}
\newcommand{\maj}{\ensuremath{^{-j}_{\m }}}
\newcommand{\paj}{\ensuremath{^{+j}_{\m }}}
\newcommand{\mai}{\ensuremath{^{-i}_{\m }}}
\newcommand{\pai}{\ensuremath{^{+i}_{\m }}}
\newcommand{\pmaj}{\ensuremath{^{\pm j}_{\m }}}
\newcommand{\mcMma}{\ensuremath{M^{-n_{\m }}_{\m }}}
\newcommand{\mcMpa}{\ensuremath{M^{+n_{\m }}_{\m }}}

\newcommand{\gr}{\operatorname{grmod}\mhyphen}
\newcommand{\Gr}{\operatorname{GrMod}\mhyphen}
\newcommand{\qgr}{\operatorname{qgrmod}\mhyphen}
\newcommand{\QGr}{\operatorname{QGrMod}\mhyphen}

\newcommand{\QgrA}{\mathcal{Q}_{\mathrm{gr}}(A)}

\newcommand{\fin}{\mathrm{fin}}

\newcommand{\Pic}{\operatorname{Pic}}

\newcommand{\s}[1]{\ensuremath{\langle #1 \rangle}}
\DeclareMathOperator{\Spec}{Spec}
\DeclareMathOperator{\MaxSpec}{MaxSpec}
\DeclareMathOperator{\Proj}{Proj}
\DeclareMathOperator{\Supp}{Supp}

\newcommand{\tors}{\ensuremath{\operatorname{tors}}\mhyphen}
\newcommand{\Tors}{\ensuremath{\operatorname{Tors}}\mhyphen}
\newcommand{\fdim}{\ensuremath{\operatorname{fdim}}\mhyphen}

\newcommand{\coh}{\operatorname{coh}}
\newcommand{\Qcoh}{\operatorname{Qcoh}}
\newcommand{\ann}{\operatorname{ann}}
\newcommand{\Usl}{U(\mathfrak{sl}_2)}

\newcommand{\Mod}{\operatorname{Mod}\mhyphen}

\newcommand{\grp}[1]{\left\langle #1 \right\rangle}

\newcommand{\red}[1]{\textcolor{red}{#1}}

\numberwithin{equation}{section}

\newtheorem{theorem}[equation]{Theorem}
\newtheorem*{theorem*}{Theorem}
\newtheorem{corollary}[equation]{Corollary}
\newtheorem{lemma}[equation]{Lemma}
\newtheorem{proposition}[equation]{Proposition}

\newtheorem{hypothesis}[equation]{Hypothesis}

\theoremstyle{remark}
\newtheorem{remark}[equation]{Remark}

\theoremstyle{definition}
\newtheorem{definition}[equation]{Definition}

\newtheorem{example}[equation]{Example}

\newtheorem*{acknowledgments}{Acknowledgments}

\newcommand\fsl{\mathfrak{sl}}
\newcommand\iso{\cong}
\newcommand\tensor{\otimes}
\newcommand\tornado{\xi}

\newcommand\cQ{\mathcal Q}

\newcommand\scC{\mathscr{C}}
\newcommand\scP{\mathscr{P}}

\usepackage{xy}

\begin{document}
\title{Simple $\mathbb{Z}$-graded domains of Gelfand--Kirillov dimension two}

\author[Ferraro]{Luigi Ferraro}
\address{Wake Forest University, Department of Mathematics and Statistics, P.O. Box 7388, Winston-Salem, North Carolina 27109} 
\email{ferrarl@wfu.edu}

\author[Gaddis]{Jason Gaddis}
\address{Miami University, Department of Mathematics, 301 S. Patterson Ave., Oxford, Ohio 45056} 
\email{gaddisj@miamioh.edu}

\author[Won]{Robert Won}
\address{University of Washington, Department of Mathematics, Box 354350, Seattle, Washington 98195} 
\email{robwon@uw.edu}

\date{}

\begin{abstract}
Let $\mathds{k}$ be an algebraically closed field and $A$ a $\mathbb{Z}$-graded finitely generated simple $\mathds{k}$-algebra which is a domain of Gelfand--Kirillov dimension 2. We show that the category of $\mathbb{Z}$-graded right $A$-modules is equivalent to the category of quasicoherent sheaves on a certain quotient stack.
The theory of these simple algebras is closely related to that of a class of generalized Weyl algebras (GWAs).
We prove a translation principle for the noncommutative schemes of these GWAs, shedding new light on the classical translation principle for the infinite-dimensional primitive quotients of $U(\mathfrak{sl}_2)$.
\end{abstract}

\maketitle

\section{Introduction}

Let $\kk$ denote an algebraically closed field of characteristic zero. Throughout this paper, all vector spaces are taken over $\kk$, all rings are $\kk$-algebras, and all categories and equivalences of categories are $\kk$-linear.

Suppose $A = \bigoplus_{i \in \NN} A_i$ is a right noetherian $\NN$-graded $\kk$-algebra. There is a notion of a noncommutative projective scheme of $A$, defined by Artin and Zhang in \cite{AZ}. Let $\Gr A$ denote the category of $\NN$-graded right $A$-modules and let $\gr A$ denote its full subcategory of finitely generated modules. 
Let $\Tors A$ ($\tors A$, respectively) denote the full subcategory of $\Gr A$ ($\gr A$, respectively) consisting of torsion modules (see Section \ref{sec:tors}). 

Let $\QGr A$ denote the quotient category $\Gr A / \Tors A$, let $\mathcal{A}$ denote the image of $A_A$ in $\QGr A$, and let $\mcS$ denote the shift operator on $\QGr A$.
The \emph{general noncommutative projective scheme of $A$} is the triple $(\QGr A, \mathcal{A}, \mcS)$.
Similarly, if we let $\qgr A = \gr A / \tors A$, then the \emph{noetherian noncommutative projective scheme of $A$} is the triple $(\qgr A, \mathcal{A}, \mcS)$.
These two categories play the same role as the categories of quasicoherent and coherent sheaves on $\Proj R$ for a commutative $\kk$-algebra $R$. By abuse of language, we will refer to either of the categories $\QGr A$ or $\qgr A$ as the \emph{noncommutative projective scheme of $A$}.

If $A_0 = \kk$, then $A$ is called \emph{connected graded}. In \cite{AS}, Artin and Stafford classified the noncommutative projective schemes of connected graded domains of Gelfand--Kirillov (GK) dimension 2. Since a connected graded domain of GK dimension 2 is a generalization of a projective curve, Artin and Stafford's theorem can be viewed as a classification of noncommutative projective curves.

\begin{theorem*}[Artin and Stafford, {\cite[Corollary 0.3]{AS}}] Let $A$ be a connected $\NN$-graded domain of GK dimension 2 which is finitely generated in degree $1$. Then there exists a projective curve $X$ such that $\qgr A \equiv \coh(X)$.
\end{theorem*}

The focus of this paper is on $\ZZ$-graded $\kk$-algebras. A fundamental example of such a ring is the (first) Weyl algebra, $A_1 = \kk\langle x, y \rangle/(yx-xy-1)$, which does not admit an $\NN$-grading but does admit a $\ZZ$-grading by letting $\deg x = 1$ and $\deg y = -1$. This $\ZZ$-grading is natural in light of the fact that $A_1$ is isomorphic to the ring of differential operators on the polynomial ring $\kk[t]$, where $x$ corresponds to multiplication by $t$ and $y$ corresponds to differentiation by $t$. In \cite{smith}, Smith showed that there exists an abelian group $\Gamma$ and a commutative $\Gamma$-graded ring $C$ such that the category of $\Gamma$-graded $C$-modules is equivalent to $\Gr A_1$. As a corollary, there exists a quotient stack $\chi$ such that $\Gr A_1$ is equivalent to the category $\Qcoh(\chi)$ of quasicoherent sheaves on $\chi$ \cite[Corollary 5.15]{smith}.

As the Weyl algebra is a domain of GK dimension 2, Smith's result can be seen as evidence for a $\ZZ$-graded version of Artin and Stafford's classification.
Further evidence is found in \cite{woncomm}, in which the third-named author studied the graded module categories over the infinite-dimensional primitive quotients $\{R_{\lambda} \mid \lambda \in \kk\}$ of $U(\mathfrak{sl}_2)$. Each $R_{\lambda}$ is a $\ZZ$-graded domain of GK dimension 2, and for each $R_{\lambda}$, there exists an abelian group $\Gamma$ and a commutative $\Gamma$-graded ring $B_{\lambda}$ such that $\QGr R_{\lambda}$ is equivalent to the category of $\Gamma$-graded $B_{\lambda}$-modules \cite[Theorems 4.2 and 4.10, Corollary 4.20]{woncomm}.

In this paper, we prove a $\ZZ$-graded analogue of Artin and Stafford's theorem for \emph{simple} domains of GK dimension 2. 
A key ingredient in our result is Bell and Roglaski's classification of simple $\ZZ$-graded domains. In \cite[Theorem 5.8]{BR}, they
proved that for any simple finitely generated $\ZZ$-graded domain $A$ of GK dimension 2, there is a generalized Weyl algebra (GWA) $A'$ satisfying Hypothesis~\ref{hyp} such that $\Gr A \equiv \Gr A'$. We are therefore able to reduce to studying these GWAs and prove the following theorem.

\begin{theorem*}[Theorem~{\ref{Bequivalent}} and Corollary~{\ref{cor.stacks}}]
Let $A = \bigoplus_{i \in \ZZ} A_i$ be a simple finitely generated $\ZZ$-graded domain of GK dimension 2 with $A_i \neq 0$ for all $i \in \ZZ$. Then there exists an abelian group $\Gamma$ and a commutative $\Gamma$-graded ring $B$ such that $\Gr A$ is equivalent to the category of $\Gamma$-graded $B$-modules. Hence, if $\chi$ is the quotient stack $\left[ \frac{\Spec B}{\Spec \kk \Gamma} \right]$ then $\QGr A \equiv \Qcoh (\chi)$. 
\end{theorem*}
We also study properties of the commutative rings $B$ that arise in the above theorem.
For instance, we prove that $B$ is non-noetherian of Krull dimension 1,
and that $B$ is coherent Gorenstein (see Section~\ref{sec.comm}).

In Section~\ref{sec.morita}, we prove a translation principle for GWAs. This builds upon the classical translation principle for $\Usl$, which we describe as follows. 
Over $\kk$, the associative algebra $\Usl$ is generated by $E,F,H$ subject to the relations $[E,F] = H$, $[H,E] = 2E$, and $[H,F] = -2F$.
The Casimir element $\Omega = 4FE + H^2 + 2H$ generates the center of $\Usl$.
The infinite-dimensional primitive factors of $U=\Usl$ are given by $R_{\lambda} = U/(\Omega - \lambda^2 + 1) U$ for each $\lambda \in \kk$.
The classical result states that $R_{\lambda}$ is Morita equivalent to $R_{\lambda + 1}$ unless $\lambda = -1,0$. We refer the reader to \cite{St2} for a proof.

The rings $R_{\lambda}$ are isomorphic to GWAs with base ring $\kk[z]$, defining automorphism $\sigma(z) = z+1$, and quadratic defining polynomial $f \in \kk[z]$. We study a class of GWAs which includes all of the $R_{\lambda}$ (see Hypothesis~\ref{hypa}).
In particular, we prove that for each GWA $A$ satisfying these hypotheses, if a new GWA $A'$ is obtained by shifting the factors of the defining polynomial of $A$ by the automorphism $\sigma$, then $\QGr A \equiv \QGr A'$ (Theorem~\ref{thm.translation}).
One consequence of this result is that although $\Mod R_{-1} \not\equiv \Mod R_{0} \not\equiv \Mod R_{1}$ and $\Gr R_{-1} \not\equiv \Gr R_{0} \not \equiv \Gr R_{1}$, nevertheless there is an equivalence
\[ \QGr R_{-1} \equiv \QGr R_{0} \equiv \QGr R_1.
\]
Hence, for any $\lambda \in \kk$, $\QGr R_{\lambda} \equiv \QGr R_{\lambda+1}$. We remark that this particular consequence follows from a more specific result proved in \cite{woncomm}, and was first observed by Sierra \cite{sierraprivate}.

\begin{acknowledgments} The authors would like to thank Daniel Chan, W. Frank Moore, Daniel Rogalski, and James Zhang for helpful conversations. We particularly thank Susan Sierra for directing us to references on Morita contexts and suggesting the relationship with quotient categories that we prove in Section~\ref{sec.morita}, Calum Spicer for his help on an earlier version of this manuscript, and S. Paul Smith for several suggestions and corrections.
\end{acknowledgments}

\section{Preliminaries}
\label{sec:bg}

In this section, we fix basic notation, definitions, and terminology which will be in use for the remainder of the paper.

\subsection{Graded rings and modules}
\label{sec:graded}

Let $\Gamma$ be an abelian semigroup. We say that a $\kk$-algebra $R$ is \emph{$\Gamma$-graded} if there is a $\kk$-vector space decomposition $R = \bigoplus_{\gamma \in \Gamma} R_{\gamma}$
such that $R_{\gamma} \cdot R_{\delta} \subseteq R_{\gamma + \delta}$ for all $\gamma, \delta \in \Gamma$. Each $R_{\gamma}$ is called the \emph{$\gamma$-graded component} of $R$ and each $r \in R_{\gamma}$ is called \emph{homogeneous of degree $\gamma$}.
Similarly, a (right) $R$-module $M$ is \emph{$\Gamma$-graded} if it has a $\kk$-vector space decomposition $M = \bigoplus_{\gamma \in \Gamma} M_{\gamma}$ such that $M_{\gamma} \cdot R_{\delta} \subseteq M_{\gamma + \delta}$ for all $\gamma, \delta \in \Gamma$. Each $M_{\gamma}$ is called the \emph{$\gamma$-graded component} of $M$.
When the group $\Gamma$ is clear from context, we will call $R$ and $M$ simply \emph{graded}.

A homomorphism $f: M \to N$ of $\Gamma$-graded right $R$-modules is called a \emph{graded homomorphism of degree $\delta$} if $f(M_{\gamma}) \subseteq N_{\gamma + \delta}$ for all $\gamma \in \Gamma$. We denote
\[ \uHom_{R}(M,N) = \bigoplus_{\delta \in \Gamma} \uHom_R(M,N)_{\delta}
\]
where $\uHom_R(M,N)_{\delta}$ is the set of all graded homomorphisms $M \to N$ of degree $\delta$. A \emph{graded homomorphism} is a graded homomorphism of degree $0$.

The $\Gamma$-graded right $R$-modules together with the graded right $R$-module homomorphisms (of degree $0$) form a category which we denote $\Gr (R,\Gamma)$. Therefore, 
\[ \Hom_{\Gr (R,\Gamma)}(M,N) = \uHom_R(M,N)_0.
\]
We use lower-case letters to denote full subcategories consisting of finitely-generated objects, so $\gr (R,\Gamma)$ denotes the category of finitely-generated $\Gamma$-graded right $R$-modules.
When $\Gamma = \ZZ$, we omit the group from our notation and refer to these categories as simply $\Gr R$ and $\gr R$.

We say two algebras $R$ and $S$ are \emph{Morita equivalent} if there
is an equivalence of categories $\Mod R \equiv \Mod S$.
If $R$ is $\Gamma$-graded and $S$ is $\Lambda$-graded for some abelian semigroup $\Lambda$, then we say that $R$ and $S$ are \emph{graded Morita equivalent} if there is a Morita equivalence between $R$ and $S$ that is implemented by graded bimodules and so also gives an equivalence between their graded module categories.

The group of autoequivalences of $\gr R$ modulo natural transformation is called the \emph{Picard group} of $\gr R$ and is denoted $\Pic(\gr R)$.
For a $\ZZ$-graded $\kk$-algebra $R$, the \emph{shift functor} is an autoequivalence of $\gr R$ which sends a graded right module $M$ to the new module $M \s{1} = \bigoplus_{j \in \ZZ} M \s{1}_j$, defined by $M \s{1}_j = M_{j+1}$. We write this functor as $\mcS_R$. We use the notation $M\s{i}$ for the module $\mcS_R^i (M)$ and note that $M\s{i}_j = M_{j+i}$. This is the standard convention for shifted modules, although it is the opposite of the convention that is used in \cite{sierra, woncomm, wonpic}.

\subsection{Noncommutative projective schemes of \texorpdfstring{$\ZZ$}{Z}-graded algebras}
\label{sec:tors}
Now suppose that $A$ is a noetherian $\ZZ$-graded $\kk$-algebra and let $M \in \Gr A$. As in \cite{smith2}, define the \emph{torsion submodule of $M$} by
\[ \tau(M) = \text{sum of the finite-dimensional submodules of $M$}.\]
The module $M$ is said to be \emph{torsion} if $\tau(M)=M$ and \emph{torsion-free} if $\tau(M)=0$.
Let $\Tors A$ ($\tors A$, respectively) denote the full subcategory of $\Gr A$ ($\gr A$, respectively) consisting of torsion modules.
This is a Serre subcategory and so we may form the quotient category $\QGr A = \Gr A/\Tors A$ ($\qgr A = \gr A / \tors A$, respectively). 

The shift functor $\mcS$ of $\Gr A$ descends to an autoequivalence of $\QGr A$ and $\qgr A$, which we also denote by $\mcS$. Let $\mathcal{A}$ denote the image of $A$ in the quotient categories. Then $(\QGr A, \mathcal{A}, \mcS)$ ($(\qgr A, \mathcal{A}, \mcS)$, respectively) is the \emph{noncommutative projective scheme} (\emph{noetherian noncommutative projective scheme}, respectively) of $A$.
When $A$ is actually $\NN$-graded, this definition coincides with the noncommutative projective scheme $\Proj A$ defined by Artin and Zhang \cite{AZ}.

It is clear that $\tors A=\fdim A$, the subcategory of $\gr A$ consisting of modules of finite $\kk$-dimension.
Since every $A$-module is a union of its finitely generated submodules, $\Tors A$ can be described as the subcategory of $\Gr A$ consisting of modules which are unions of their finite-dimensional submodules.

\subsection{\texorpdfstring{Generalized Weyl algebras and $\ZZ$-graded simple rings}{Generalized Weyl algebras and Z-graded simple rings}}
\label{sec.gwadef}
Let $R$ be a ring, let $\sigma: R \to R$ an automorphism of $R$, and fix a central element $f \in R$.
The \emph{generalized Weyl algebra (GWA) of degree one} $A=R(\sigma,f)$ is the quotient of $R \langle x, y \rangle$ by the relations
\[xy = f,  \quad  yx = \sigma^{-1}(f), \quad  xr = \sigma(r)x,   \quad yr = \sigma^{-1}(r)y \]
for all $r \in R$.
We call $R$ the \emph{base ring} and $\sigma$ the \emph{defining automorphism} of $A$.
Generalized Weyl algebras were so-named by Bavula \cite{bav}, and many well-studied rings can be realized as GWAs, including the classical Weyl algebras, ambiskew polynomial rings, and generalized down-up algebras. 
There is a $\ZZ$-grading on $R(\sigma, f)$ given by $\deg x = 1$, $\deg y = -1$, and $\deg r = 0$ for all $r \in R$. 
We also remark that if $R$ is commutative, then every $\ZZ$-graded right $A$-module $M$ has an $(R,A)$-bimodule structure as follows: if $m \in M$ is homogeneous of degree $i$, then the left $R$-action of $r \in R$ is given by
\[ r \cdot m = m \cdot \sigma^{-i}(r).
\]

In this paper, every GWA satisfies the following hypothesis.
\begin{hypothesis}\label{hypa} Let $A = R(\sigma, f)$ be the GWA with 
\begin{enumerate}
\item base ring $R=\kk[z]$ and defining automorphism $\sigma(z) = z+1$, or
\item \label{case2} base ring $R=\kk[z,z\inv]$ and defining automorphism $\sigma(z) = \xi z$ for some nonroot of unity $\xi \in \kk^\times$. 
\end{enumerate}
Assume that $f \in \kk[z]$ is monic and in case~\ref{case2} assume that $0$ is not a root of $f$. Let $\Z$ denote the set of roots of $f$ and for $\m \in \Z$, let $n_\m $ denote the multiplicity of $\m $ as a root of $f$.  
\end{hypothesis}
For any $\eta \in R^\times$, there is an isomorphism $A \cong R(\sigma, \eta f)$ mapping $x$ to $\eta x$, $y$ to $y$, and $z$ to $z$.
Hence, by adjusting by an appropriate unit in $R$, every GWA with base ring and defining automorphism as above is isomorphic to one satisfying the additional assumptions in Hypothesis~\ref{hypa}.

Since we are assuming that $R = \kk[z]$ or $\kk[z,z\inv]$, the GWA $A$ is a noetherian domain \cite[Proposition 1.3]{bav} of Krull dimension one \cite[Theorem 2]{bavkdim}. 
Hence, $A$ is an Ore domain. We denote by $\cQ(A)$ the quotient division ring of $A$, obtained by localizing $A$ at all nonzero elements. The \emph{rank} of an $A$-module $M$ is the dimension of $M \otimes_A \cQ(A)$ over the division ring $\cQ(A)$. 
Since $A$ is $\ZZ$-graded, we also consider $\QgrA$, the  graded quotient division ring of $A$, obtained by localizing $A$ at all nonzero homogeneous elements.
The field of fractions of $A_0 = R$ is $\kk(z)$, so $\QgrA$ is a skew Laurent ring over $\kk(z)$:
\begin{align}
\label{eq.qgrA}
\QgrA = \kk(z)[x,x^{-1}; \sigma] = \bigoplus_{i \in \ZZ} \kk(z)x^i.
\end{align}

We use the notation $\sigma_{\kk}$ to denote the action of $\sigma$ on the $\kk$-points of $\Spec R$; i.e., for $\lambda \in \kk$, if $\sigma(z) = z +1$ then $\sigma_{\kk}(\lambda) = \lambda - 1$ and if $\sigma(z) = \xi z$ then $\sigma_{\kk}(\lambda) =  \xi\inv \lambda$.
We say two roots of $f$ are \emph{congruent} if they are on the same $\sigma_\kk$-orbit. For a GWA $A$ satisfying Hypothesis~\ref{hypa}, by \cite{hodges, bavgldim}, the global dimension of $A$ depends only on the roots of $f$:
\[ \gldim A = \begin{cases} 
1 & \text{if $f$ has no multiple roots and no congruent roots} \\
2 & \text{if $f$ has a congruent root but no multiple roots} \\
\infty & \text{if $f$ has multiple roots.}
\end{cases}
\]
It follows from \cite[Theorem 3]{bavkdim} that $A$ is simple if and only if no two distinct roots of $f$ are congruent.

In \cite{BR}, Bell and Rogalski showed that simple $\ZZ$-graded domains are closely related to GWAs satisfying Hypothesis~\ref{hypa}.

\begin{theorem}[Bell and Rogalski, {\cite[Theorem 5.8]{BR}}]Let $S = \bigoplus_{i \in \ZZ} S_i$ be a simple finitely generated $\ZZ$-graded domain of GK dimension 2 with $S_i \neq 0$ for all $i \in \ZZ$. Then $S$ is graded Morita equivalent to a GWA $R(\sigma,f)$ satisfying Hypothesis~\ref{hypa} where additionally no two distinct roots of $f \in R$ are congruent. 
\end{theorem}

Hence, if we are interested only in the category of graded modules over simple $\ZZ$-graded rings of GK dimension 2, it suffices to restrict our attention to these GWAs. Starting in Section~\ref{sec.grmods}, we operate under the following additional hypothesis:
\begin{hypothesis}\label{hyp} Let $A = R(\sigma, f)$ be a GWA satisfying Hypothesis~\ref{hypa}. Further assume that $f \in \kk[z]$ has no two distinct roots on the same $\sigma_\kk$-orbit so that $A$ is simple. 
\end{hypothesis}

Thus, when we restrict to Hypothesis~\ref{hyp},
then it follows that $\gldim A \neq 2$.

\section{A translation principle for GWAs}
\label{sec.morita}

Throughout, suppose $A$ is a GWA satisfying Hypothesis~\ref{hypa}. 
The simple graded right $A$-modules were described by Bavula in \cite{bavkdim} and we adopt his terminology here. 
The group $\grp{\sigma}$ acts on $\MaxSpec R$, the set of maximal ideals of $R$. Specifically, the orbit of $(z - \lambda) \in \MaxSpec R$ is given by 
\[\cO_\lambda = \left\{ \sigma^{i}(z - \lambda) \st i \in \ZZ \right\} =  \left\{ (z - \sigma_{\kk}^i(\lambda)) \st i \in \ZZ \right\}.
\]
If the $\sigma$-orbit of $(z-\lambda)$ contains no factors of $f$, it is called \textit{nondegenerate}, otherwise it is called \textit{degenerate}. If two distinct factors $(z-\lambda)$ and $\sigma^i(z-\lambda)$ lie on the same $\sigma$-orbit, then $\lambda$ and $\sigma_{\kk}^i(\lambda)$ are congruent roots of $f$ and we call $\cO_{\lambda}$ a \textit{congruent orbit}. Otherwise, a degenerate orbit is called a \textit{non-congruent orbit}.


\begin{lemma}[{\cite[Theorem 1]{bavkdim}}] \label{lem.simples} Let $A = R(\sigma, f)$. The simple modules of $\gr A$ are given as follows.
\begin{enumerate}
\item For each nondegenerate orbit $\cO_\lambda$ of $\MaxSpec R$, one has the module 
$\displaystyle M_\lambda = \frac{A}{(z-\lambda) A}$ and its shifts $M_{\lambda}\s{n}$ for each $n \in \ZZ$.

\item For each degenerate non-congruent orbit $\cO_{\m }$, one has
\begin{enumerate}
\item the module $\displaystyle M\ma = \frac{A}{(z - \m )A + xA}$ and its shifts $M^{-}_{\m } \s{n}$ for each $n \in \ZZ$ and
\item the module $\displaystyle M\pa = \frac{A}{\sigma^{-1}(z - \m )A + yA} \s{-1}$ and its shifts $M^{+}_{\m } \s{n}$ for each $n \in \ZZ$.
\end{enumerate}

\item For each degenerate congruent orbit $\cO_{\m }$, label the roots on the orbit so that they are given by $\m, \sigma_{\kk}^{i_1}(\m ), \dots, \sigma_{\kk}^{i_r}(\m )$ where $0 > i_1 > \dots > i_r$. Set $i_0 = 0$. Then one has
\begin{enumerate}
\item the module $\displaystyle M\ma = \frac{A}{(z - \m )A + xA}$ and its shifts $M^{-}_{\m } \s{n}$ for each $n \in \ZZ$,
\item for each $k = 1, \dots, r$, the module \[M_{\m }^{(i_{k-1}, i_k]} = \frac{A}{\sigma^{i_k}((z- \m ))A + xA + y^{i_k- i_{k-1}} A} \s{i_k}\] and its shifts $M_{\m }^{(i_{k-1}, i_k]} \s{n}$ for each $n \in \ZZ$, and
\item the module $\displaystyle M\pa = \frac{A}{\sigma^{i_r-1}(z - \m )A + yA} \s{i_r -1}$ and its shifts $M^{+}_{\m } \s{n}$ for each $n \in \ZZ$.
\end{enumerate}
\end{enumerate}
\end{lemma}

The notation $M\pma$ is intended to reflect the fact that $M\ma$ is nonzero only in sufficiently negative degree while $M\pa$ is nonzero only in sufficiently positive degree. The module $M_{\m }^{(i_{k-1},i_k]}$ is nonzero in degrees $(i_{k-1},i_k] \cap \ZZ$. 

\begin{example}If $R = \kk[z]$, $\sigma(z) = z+1$, and $f = z(z-1)^2(z-3)$, then there is a single degenerate congruent orbit $\cO_0$ with roots $0, \sigma_{\kk}^{-1}(0) =  1, \sigma_{\kk}^{-3}(0) = 3$. In this case, the simple modules are given by the shifts of $M_{\lambda} = A/(z - \lambda) A$ for each $\lambda \in \kk \setminus \ZZ$ as well as all shifts of the modules 
\begin{itemize}
    \item $\displaystyle M_0^- = \frac{A}{zA + xA}$,
    \item $\displaystyle M_0^{(0,1]} = \frac{A}{(z-1)A + xA + yA} \s{-1}$,
    \item $\displaystyle M_0^{(1,3]} = \frac{A}{(z-3)A + xA + y^2A} \s{-3}$, and
    \item $\displaystyle M_0^+ = \frac{A}{(z-4)A + xA + yA} \s{-4}$.
\end{itemize}
These modules are exactly the simple modules appearing in the composition series of the module $A/zA$.
\end{example}

As described in the introduction, Stafford proved a translation principle for the infinite dimensional primitive factors $R_{\lambda}$ of $U(\fsl_2)$ \cite{St2}. These rings are the GWAs $\kk[z](\sigma,f)$ of type (1) in Hypothesis~\ref{hypa} with quadratic defining polynomial $f = z(z-\lambda) \in \kk[z]$.
Stafford showed that for $\lambda \in \kk$, the rings $\kk[z](\sigma,z(z-\lambda))$ and $\kk[z](\sigma,z(z-\lambda-1))$ are Morita equivalent unless $\lambda = 0$ or $-1$. 
The cases $\lambda=0,-1$ correspond to when the first or second GWA, respectively, has infinite global dimension. 
Hodges also studied Morita equivalences between these rings, and used $K$-theoretic techniques, along with Stafford's result, to prove that two primitive factors $R_\lambda$ and $R_{\mu}$, are Morita equivalent if and only if $\lambda\pm\mu \in \ZZ$ \cite{hodges2}.

Several authors have studied Morita equivalences between GWAs  $R(\sigma,f)$ satisfying Hypothesis~\ref{hypa}(1) with higher degree defining polynomials $f$. Jordan gave sufficient conditions on $f$ and $f'$ for $R(\sigma, f)$ and $R(\sigma, f')$ to be Morita equivalent \cite[Lemma 7.3(iv)]{jordan2}. Richard and Solotar gave necessary conditions for the Morita equivalence of $R(\sigma, f)$ and $R(\sigma,f')$ with the additional hypotheses that the GWAs are simple and have finite global dimension \cite{richardsolotar}.
Shipman fully classified the equivalence classes of GWAs satisfying Hypothesis~\ref{hypa}(1) under the stronger notion of strongly graded Morita equivalences \cite{shipman}.

Our treatment builds upon \cite{St2}. We extend Stafford's translation principle to all GWAs satisfying Hypothesis~\ref{hypa}. 
This was essentially done by Jordan \cite[Lemma 7.3(iv)]{jordan2}---our contribution is to pay close attention to
those situations in which Stafford's techniques do \emph{not} give Morita equivalences.
In these cases, we show that Stafford's methods still yield a graded Morita context.
Using this Morita context, we prove that there is an equivalence between the \emph{quotient categories} which play the role of the noncommutative projective schemes for these $\ZZ$-graded rings. Hence, the results of this section can be viewed as a translation principle for the noncommutative projective schemes of GWAs, which is analogous to Van den Bergh's translation principle for central quotients of the four-dimensional Sklyanin algebra \cite{vdb3}. 

Before we proceed, we recall some details on graded Morita contexts and the graded Kato--M\"{u}ller Theorem in the $\ZZ$-graded setting.
Recall that a \emph{Morita context} between two algebras $T$ and $S$ is a 6-tuple $(T,S,{}_S M_T, {}_T N_S, \phi, \psi)$ where ${}_SM_T$ and ${}_TN_S$ are bimodules and $\phi:N \tensor_S M \to T$ and $\psi:M \tensor_T N\to S$ are bimodule morphisms satisfying 
\begin{itemize}
    \item $\psi(m \tensor n)m' = m\phi(n \tensor m')$
    \item $\phi(n \tensor m)n' = n\psi(m \tensor n')$
\end{itemize}
for all $m,m' \in M$, $n,n' \in N$. We refer to $I=\im\phi$ and $J=\im\psi$ as the \emph{trace ideals} of the Morita context.
If both $\phi$ and $\psi$ are surjective, then this Morita context gives a Morita equivalence between $T$ and $S$ and so $\Mod T \equiv \Mod S$.

Let $T$ be an algebra and let $M$ be a right $T$-module. One important construction of a Morita context is given by $(T, S, M, M^*, \phi, \psi)$ where $M^* = \Hom_T(M,T)$, $S = \Hom_T(M,M)$, and where $\phi: M^* \otimes_S M \to T$ and $\psi: M \otimes_T M^* \to S$ are defined by $\phi(f \otimes m) = f(m)$ and $\psi(m \otimes f)(m') = mf(m')$ for $f \in M^*$ and $m,m' \in M$.
This Morita context is a Morita equivalence if and only if $M$ is a progenerator of $\Mod T$.

This theory carries over to the graded setting (see \cite[Chapter 2]{hazrat}). 
Let $\Gamma$ be an abelian semigroup.
If $T$ is a $\Gamma$-graded algebra and $M$ is a finitely generated graded right $A$-module, then both $M^* = \Hom_T(M,T)$ and $S = \Hom_T(M,M)$ are, in fact, graded with $M^* =  \uHom_T(M,T)$ and $S = \uHom_T(M, M)$. Then $(T, S, M, M^*, \phi, \psi)$ is a \emph{graded Morita context}. If $M$ is a graded progenerator then this graded Morita context gives a graded Morita equivalence so $\Mod T \equiv \Mod S$ and $\Gr T \equiv \Gr T$. 
For any two-sided graded ideal $I$ of $T$, the class
\[\scP_I=\{M \in \Gr T \mid MI=0\}\]
is a rigid closed subcategory of $\Gr T$.
Let $\scC_I$ be the smallest localizing subcategory
of $\Gr T$ containing $\scP_I$.
The following is a graded version of the Kato--M\"{u}ller Theorem \cite{kato,muller}.

\begin{theorem}[Iglesias and N\v{a}st\v{a}sescu \cite{IN}]
Let $(T,S,{}_S M_T, {}_T N_S, \phi, \psi)$ be a $\Gamma$-graded Morita context between $\Gamma$-graded rings $T$ and $S$. Set $I=\im\phi$ and $J=\im\psi$. The graded functors
\begin{align*}
\uHom_T(M,-):\Gr T \rightleftarrows \Gr S : \uHom_S(N,-)
\end{align*}
induce functors between the quotient categories
\begin{align}\label{eq.KM}
\uHom_T(M,-):\Gr T/\scC_I \rightleftarrows \Gr S/\scC_J: \uHom_S(N,-).
\end{align}
\end{theorem}

Note that $\uHom_T(M,-)$ intertwines the shift functors of $\Gr T$ and $\Gr S$. In particular, if $X$ is a graded $T$-module, then $\uHom_T(M,\mcS_T X) = \mcS_S \uHom_T(M,X)$ \cite[Lemma 2.2]{GG}.
This result also descends to the quotient categories: the functor induced by $\uHom_T(M,-)$ commutes with the functors induced by the shift functors of $\Gr T$ and $\Gr S$. 

Let $A = R(\sigma,f)$ be a GWA satisfying Hypothesis~\ref{hypa} with $n = \deg(f) > 0$.
Let $h \in \kk[z]$ be a non-constant factor of $f$ and define $g=h\inv f$. Define the graded module
\[ M = hA+xA\]
and construct a graded Morita context 
\begin{align}\label{eq.morita}
(A,B,{}_BM_A,{}_A{M^*}_B,\phi,\psi) \tag{$\dagger$}
\end{align}
as discussed above with $M^* = \uHom_{A}(M,A)$ and $B = \uHom_A(M,M)$.
We make the identification 
\[M^*=\{p \in \QgrA \mid pM \subset A\}\] 
where $\QgrA = \kk(z)[x,x\inv;\sigma]$ is the graded quotient division ring of $A$ \cite[Proposition 3.1.15]{mcconnell}. 
Similarly, we make the identification
\[B = \{p \in \QgrA \mid pM \subset M\}.\]
Under this identification, the graded bimodule maps 
$\phi:M^* \tensor_B M \to A$ and $\psi:M \tensor_A M^* \to B$ are given by $\phi(m',m) = m'm$ and $\psi(m,m')=mm'$ for $m \in M$ and $m' \in M'$.
Since $M$ is not a graded progenerator in general, this graded Morita context is not always a graded Morita equivalence. However, we will show that it induces a graded equivalence between certain quotient categories of $\Gr A$ and $\Gr B$.
Part (3) of the lemma below recovers \cite[Corollary 3.3]{St2} when $R = \kk[z]$ and $f \in \kk[z]$ is quadratic.

\begin{lemma}
\label{lem.mcontext}
Retain notation as above. 
Suppose $\beta$ is a root of $f$ of multiplicty one.
Set $h=z-\beta$ and set $d=\gcd(\sigma\inv(g),h)$ in $A$.
Let $A'=R(\sigma,f')$ be the GWA with $f'=\sigma(h)g$.
Then
\begin{enumerate}
\item $B \iso A'$ so \eqref{eq.morita} gives a graded Morita context between $A$ and $A'$,
\item $M$ is a projective right $A$-module, and
\item $M$ is a generator for $A$ if and only if $d = 1$.
\end{enumerate}
\end{lemma}

\begin{proof}
(1) It is clear from the identification above that $A \subset M^*$.
Observe that $x\inv gx = \sigma\inv(g) \in A$ and
$x\inv g h = x\inv f = x\inv xy = y \in A$.
Hence, $x\inv g \in M^*$ and so $A+Ax\inv g \subset M^*$.

Note that $MM^* \subset B$. Clearly $1, x \in B$.
Since $zx=x\sigma\inv(z) \in xA$, then $z \in B$.
As $x\inv g \in M^*$, then
$x\inv \sigma(h) g = h x\inv g \in MM^* \subset B$.
Let $S$ be the subalgebra of $B \subset \QgrA$ generated by $X=x$, $Y=x\inv \sigma(h) g$, and $Z=z$.
We check that these generators satisfy the relations for $A'$.
It is clear that $XZ=\sigma(Z)X$ and $XY=f'$. A computation shows that
\begin{align*}
YZ &= x\inv \sigma(h)gz = \sigma\inv(z)x\inv\sigma(h)g = \sigma\inv(Z)Y, \\
YX &= x\inv \sigma(h)gx = h\sigma\inv(g) = \sigma\inv(f'),
\end{align*}
and so $S \iso A'$.

Given any $q \in M$, $Bq \subset M \subset S$.
By \cite[Corollary 4.4]{jordan2}, $S$ is a maximal order in its quotient ring and by \cite[Lemma 2]{VV}, it is also a 
\emph{graded} maximal order in its \emph{graded} quotient ring, whence $B=S\iso A'$.

(2) By the Dual Basis Lemma, $M$ is projective if and only if $1 \in MM^*$ \cite[Lemma 3.5.2]{mcconnell}.
By the proof of (1), $A+Ax\inv g \subset M^*$.
Thus, it is clear that $X,h(Z) \in \im\psi$.
We have $x(x\inv g) = g(Z) \in \im\psi$, so $1 = \gcd(g(Z),h(Z)) \in \im\psi$.

(3) We claim $M^*=A+Ax\inv g$. Let $p \in \QgrA$ such that $pM \subset A$. By \eqref{eq.qgrA}, we may write $p = \sum_{i=n}^m x^i p_i(z)$ with $p_i(z) \in \kk(z)$. Since $px \in A$, then $n\geq -1$ and $p_i(z) \in \kk[z]$ for all $i \geq 1$. 
On the other hand, suppose $p_{-1}(z) \neq 0$. Since $x^ip_i(z)h \in A$ for all $i\geq 0$, then $x^{-1}p_{-1}(z)h \in A$ and so $g \mid p_{-1}(z)$. The claim follows.

Thus, $\im\phi$ is generated by products of generators in $M^*$ and $M$. Since $A \subset M^*$, then $x,h \in M^*M$.
Furthermore, $y = (x\inv g)h \in M^*M$
and $\sigma\inv(g) = (x\inv g)x \in M^*M$, so $d \in \im\phi$.
Since $(xA)_0 = f R$ and $(yA)_0 = \sigma\inv(g)R$, therefore $(M^* M)_0$ is the ideal of $R$ generated by $h$ and $\sigma\inv(g)$. Hence, $1 \in \im\phi$ if and only if $d = 1$.
\end{proof}

We are now prepared to prove the translation principle for GWAs which is the main result of this section.
The result in (1) below should also be compared to a result of Shipman \cite{shipman} in the strongly graded Morita equivalence setting.

\begin{proposition}
\label{prop.collapse}
Let $A=R(\sigma,f)$. Choose a root $\beta_1$ of $f$, and consider its $\sigma_{\kk}$-orbit, $\cO$. Write $f=(z-\beta_1)\cdots(z-\beta_n) \cdot \tilde{f}$ where $\beta_1, \dots, \beta_n \in \cO$ and $\tilde{f}$ has no roots on $\cO$.
For each $1 \leq j \leq n$, 
write $\beta_j=\sigma_{\kk}^{-i_j}(\beta_1)$ for some $i_j \in \ZZ$.
After reordering, assume that $i_j \geq i_{j-1}$ for $1 < j \leq n$.
\begin{enumerate}
\item Suppose that either $k=1$ or $2 \leq k \leq n$ and $i_k-i_{k-1}>1$. Suppose further that $n_{\beta_k}=1$. Let $A' = R(\sigma, f')$ where $f'$ is the same as $f$ except that the factor $(z-\beta_k)$ of $f$ has been replaced by a factor of $(z-\sigma_\kk(\beta_k))$ in $f'$.
Then there is an equivalence of categories $\Gr A \equiv \Gr A'$.

\item Suppose that either $k=n$ or $1 \leq k \leq n-1$ and $i_{k+1}-i_k>1$. Suppose further that $n_{\beta_k}=1$.
Let $A' = R(\sigma, f')$ where $f'$ is the same as $f$ except that the factor $(z-\beta_k)$ of $f$ has been replaced by a factor of $(z-\sigma_\kk\inv(\beta_k))$ in $f'$.
Then there is an equivalence of categories $\Gr A \equiv \Gr A'$.

\item Suppose for some $2 \leq k \leq n$ that $i_k-i_{k-1}=1$, and that $n_{\beta_k}=1$. 
Let $A' = R(\sigma, f')$ where $f'$ is the same as $f$ except that the factor $(z-\beta_k)$ of $f$ has been replaced by a factor of $(z- \beta_{k-1})$ in $f'$.
Then there is an equivalence of categories $\Gr A/\scC_I \equiv \Gr A'$ where $I=(x,y,z-\beta_k)$. Therefore, $\QGr A \equiv \QGr A'$. 
\end{enumerate}
\end{proposition}

\begin{proof}
(1) Set $h(z)=z-\beta_k$. Then by hypothesis, $\gcd(\sigma\inv(g),h)=1$ and so Lemma~\ref{lem.mcontext} (2) and (3) imply that $M$ is a progenerator for $A$. Thus the Morita context in (1) is an equivalence.

(2) By the hypotheses, $\sigma_\kk\inv(\beta_k)$ has multiplicity one as a root of $f'$ and so we may apply (1) to obtain $\Gr A' \equiv \Gr A$.

(3) Again set $h(z)=z-\beta_k$ so that $\gcd(\sigma\inv(g),h)=h$ and $\gcd(g,h)=1$. Thus, by Lemma \ref{lem.mcontext} (3), $I=\im\phi=xA+yA+hA$ is the trace ideal corresponding to $\phi$. By Lemma \ref{lem.mcontext} (2), $1 \in J=\im\psi$, and hence $\scC_J=0$.
Then $\Gr A/\scC_I \equiv \Gr A'$ by the graded Kato--M{\"u}ller Theorem \cite[Theorem 4.2]{IN}.
Since $I$ is a trace ideal, it is idempotent, so $\scC_I=\scP_I$ by \cite[Remark 2.5]{IN}. To arrive at the intended equivalence, it remains to prove that $\scC_I$ contains only torsion modules.

We apply \cite[Theorem 4.2]{IN} and our description of the finite-dimensional simple modules.
It is clear that $I$ is the annihilator of $M^{(i_{k-1},i_k]}$ (and its shifts) but that $\ann_N I = \{ n \in N : nx=0 \text{ for all } x \in I\} = 0$ for any other simple $A$-module $N$, including those corresponding to nondegenerate orbits. 



As $M$ is a finitely generated graded right $A$-module, there is a graded surjection from the graded free module
$\bigoplus_{i=1}^n A\s{d_i} \to M$ which induces a graded injection \[\uHom_A(M,N) \to \uHom_A\left(\bigoplus_{i=1}^n A\s{d_i},N\right) = \bigoplus_{i=1}^n N\s{-d_i}.\]
Thus, $\uHom_A(M,N)$ is isomorphic to a subspace of $\bigoplus_{i=1}^n N\s{-d_i}$, which is finite-dimensional if and only if $N$ is finite-dimensional.
Therefore, $\uHom_A(M,-)$ maps finite-dimensional modules to finite-dimensional modules.
Since it also preserves submodule inclusion, $\uHom_A(M,-)$ also maps torsion modules to torsion modules.

Finally, let $M \in \scC_I$ and let $m \in M$. Since $mI=0$ and $A/I$ is finite-dimensional, then $\dim_\kk(Am)<\infty$. It follows that $M \in \Tors A$, so $\scC_I$ is a subcategory of $\Tors A$, whence the equivalence $\Gr A/\scC_I \equiv \Gr A'$ induces an equivalence $\QGr A \equiv \QGr A'$.
\end{proof}

The following example illustrates our methods from Proposition~\ref{prop.collapse}.

\begin{example}
Let $A=\kk[z,z\inv](\sigma,f)$ where $\sigma(z)=pz$ for some nonroot of unity $p \in \kk^\times$ and $f=(z-1)(z-p)^2(z-p^2)$.
Let $A'=\kk[z,z\inv](\sigma,f')$ with $\sigma$ as before but $f'=(z-1)(z-p)(z-p^2)(z-p^3)$. Applying Proposition~\ref{prop.collapse} (3), we can replace $f'$ with $f'=(z-1)(z-p)^2(z-p^3)$ up to equivalence of the quotient category $\QGr A$. Then applying Proposition~\ref{prop.collapse} (1) to the last factor we have that $\QGr A \equiv \QGr A'$. 
On the other hand, let $A''=\kk[z,z\inv](\sigma,f'')$ with $\sigma$ as before but $f''=(z-1)^4$. We can collapse $f'$ to $f''$ in the following way, where the labels on the arrows indicate the applicable part of Proposition~\ref{prop.collapse}:
\begin{align*}
f'  &\overset{(3)}{\rightsquigarrow} (z-1)^2(z-p^2)(z-p^3)
    \overset{(1)}{\rightsquigarrow} (z-1)^2(z-p)(z-p^3)
    \rightsquigarrow \cdots \\
    \cdots & \overset{(3)}{\rightsquigarrow} (z-1)^3(z-p^3)
    \overset{(1)}{\rightsquigarrow} (z-1)^3(z-p)
    \overset{(3)}{\rightsquigarrow} (z-1)^4 = f''.
\end{align*}
Hence, $\QGr A' \equiv \Gr A''$ and, transitively,
$\QGr A \equiv \Gr A''$.
\end{example}

Translating a root of multiplicity one was accomplished in Proposition \ref{prop.collapse}. The general case of translating a multiple root is 
our main result below.

\begin{theorem} \label{thm.translation} Let $A = R(\sigma, f)$ be a GWA satisfying Hypothesis~\ref{hypa}. Let $A' = R(\sigma, f')$ where $f'$ is obtained from $f$ by replacing any irreducible factor $(z-\beta)$ of $f$ by $\sigma(z-\beta)$ in $f'$. Then $\QGr A \equiv \QGr A'$.
\end{theorem}
\begin{proof}
Let $\cO \subset \kk$ be the set of roots of $f$ on the $\sigma$-orbit of $\beta$. Let $f_1 = \prod_{\alpha \in \cO} (z-\alpha)$ and let $f_2=ff_1\inv$. Suppose $f_1$ has degree $n$. Write
\[ f_1 = (z-\alpha_1)^{n_1} \cdots (z-\alpha_k)^{n_k}\]
where $\sum n_i = n$, $\alpha_i = \sigma_{\kk}^{-j_i}(\alpha_1)$ for $i>1$, and $j_k > j_{k-1} > \cdots > j_2>0$.
Assume $\alpha_\ell = \beta$ and write
\[ \tilde{f} =  (z-\alpha_1)^{n_1} \cdots (z-\alpha_{\ell-1})^{n_{\ell-1}}(z-\beta_1)(z-\beta_2) \cdots (z-\beta_t)\]
where $t=n-(n_1 + \cdots + n_{\ell-1})$, $\beta_1=\beta$, and $\beta_i = \sigma_{\kk}^{-i}(\beta_1)$ for $i>1$.
Let $\tilde{A}=R(\sigma,\tilde{f}f_2)$. 
We claim that both $\QGr A$ and $\QGr A'$ are equivalent to $\QGr \tilde{A}$.

The roots $\beta_i$ all have multiplicity one in $\tilde{f}$ and so we can translate them, one at time, using Proposition \ref{prop.collapse}. 
In what follows, all replacements are up to equivalence in the quotient category $\QGr \tilde{A}$.
First, we use Proposition \ref{prop.collapse} (3) to replace $(z-\beta_1)(z-\beta_2)$ with $(z-\beta_1)^2$. Suppose we have replaced $(z-\beta_1)\cdots(z-\beta_s)$ with $(z-\beta_1)^s$ for some $s < n_\ell$. Applying Proposition \ref{prop.collapse} (1) $s-2$ times to the factor $(z-\beta_{s+1})$ replaces it with $(z-\beta_2)$. We can then apply Proposition \ref{prop.collapse} (3) again to replace $(z-\beta_1)^s(z-\beta_2)$ with $(z-\beta_1)^{s+1}$.
Hence, by induction, we have
\[ \tilde{f} \rightsquigarrow (z-\alpha_1)^{n_1} \cdots (z-\alpha_{\ell})^{n_{\ell}}(z-\beta_{n_{\ell}+1})\cdots (z-\beta_t).\]
We apply Proposition \ref{prop.collapse} (1) to replace $(z-\beta_{n_{\ell}+1})$ by $(z-\alpha_{\ell+1})$. Applying Proposition \ref{prop.collapse} (1) and (3) alternatively, we may then collapse the roots  $\beta_{n_{\ell}+2},\hdots,\beta_{n_{\ell}+n_{\ell+1}}$ to $\alpha_{\ell+1}$. Continuing in this way we obtain $\QGr \tilde{A} \equiv \QGr A$.

On the other hand, we can shift or collapse $(z-\beta_1)$ to $\sigma(z-\beta_1)$ and then repeat as above to shift the remaining $\beta_i$ to their corresponding $\alpha_j$. In this way we obtain $\QGr \tilde{A} \equiv \QGr A'$. The result follows.
\end{proof}

Theorem \ref{thm.translation} says, essentially, that the $\operatorname{QGrMod}$-equivalence class of $R(\sigma,f)$ depends only on the number of roots on the same $\sigma$-orbit.

\begin{corollary}
\label{cor.qgr}
Suppose $A = R(\sigma,f)$ is a GWA satisfying Hypothesis~\ref{hypa}. Then $\QGr A \equiv \Gr A'$ for some simple GWA $A' = R(\sigma,f')$ satisfying Hypothesis~\ref{hyp}. 
\end{corollary}
\begin{proof}
Repeatedly apply Theorem~\ref{thm.translation} to each distinct $\sigma$-orbit of the irreducible factors of $f$ until the resulting polynomial $f'$ has no distinct irreducible factors on the same $\sigma$-orbit.
Then $f'$ has no congruent roots so $A'$ is simple and $\QGr A \equiv \QGr A'$. Since $A'$ is simple, $\Gr A'$ has no finite-dimensional modules, and hence $\QGr A \equiv \QGr A' = \Gr A'$.
\end{proof}

\section{Graded modules over simple GWAs}
\label{sec.grmods}

Henceforth we assume Hypothesis~\ref{hyp} unless otherwise stated.
In \cite[Section 3]{wonpic}, the third-named author studied the finite-length modules and projective modules in $\gr R(\sigma, f)$ when $R = \kk[z]$ and $f \in \kk[z]$ is quadratic. Here, we prove similar results for the simple GWAs of interest, thereby generalizing these results in two directions: first, by considering both base rings $R = \kk[z]$ and $R = \kk[z,z\inv]$ and second, by considering polynomials of arbitrary degree with no two distinct roots on the same $\sigma_{\kk}$-orbit.
In contrast with the treatment in \cite{wonpic}, rather than trying to understand all extensions of simple graded modules, we will focus only on certain indecomposable modules, which we will need in Section~\ref{sec:pic}.

\subsection{Finite-length indecomposable modules}
\label{sec:simple}

Recall the description of the simple graded $A$-modules from Lemma~\ref{lem.simples}.
Since we are assuming Hypothesis~\ref{hyp}, there are no degenerate congruent orbits, so the third case in Lemma~\ref{lem.simples} does not occur. Hence, $\Gr A$ has no modules of finite $\kk$-dimension, so throughout this section, $\Gr A = \QGr A$.

By \cite[Lemma 3.6]{bavgldim}, both $M\ma$ and $M\pa$ have projective dimension $1$ if $n_{\m } = 1$ and infinite projective dimension if $n_\m > 1$. The following lemma shows that there are  indecomposable modules of projective dimension $1$ and length $n_\m $whose composition factors are all $M \ma$ or all $M\pa$.

\begin{lemma} \label{lem.indec}For each $\m \in \Z$ and each $j = 1, \dots, n_{\m }$, there is an indecomposable graded right $A$-module of length $j$
\[M\maj = \frac{A}{(z-\m )^{j}A + xA}
\]
whose composition factors are $j$ copies of $M\ma$. This is the unique indecomposable module of length $j$ whose composition factors are $M\ma$. Similarly, there is a unique indecomposable graded right $A$-module of length $j$
\[M\paj = \frac{A}{\sigma^{-1}\left((z-\m )^{j}\right)A + yA} \s{-1}
\]
whose composition factors are $j$ copies of $M\pa$. Further, both $\mcMma$ and $\mcMpa$ have projective dimension $1$.
\end{lemma}
\begin{proof} We prove the statements for $M\ma$. The proofs of the statements for $M\pa$ are similar.

We proceed by induction on $j$. For the base case, $M_{\m }^{-1} = M\ma$. For $2 \leq j \leq n_{\m }$, consider the natural projection $M\maj \sra M_{\m }^{-(j-1)}$ whose kernel is given by
\[ K = \frac{(z-\m )^{j-1}A + xA}{(z-\m )^{j} A + xA} \cong \frac{(z-\m )^{j-1}A}{(z-\m )^{j-1} A \cap \left[(z-\m )^{j} A + xA\right]}.
\]
There is an isomorphism $M\ma \cong K$ given by left multiplication by $(z-\m )^{j-1}$. Hence, $M\maj$ is a length $j$ module whose composition factors are all isomorphic to $M\ma$.

To show that $M\maj$ is indecomposable, it suffices to show that $M\maj$ has a unique submodule isomorphic to $M\ma$.
Any homomorphism $\varphi: M\ma \sra M\maj$ is determined by $\varphi(1) \in (M\maj)_0$. Since, in $M\maj$, $(z-\m  )^j = 0$, every element of $(M\maj)_0$ can be written as $\sum_{\ell = 0}^{j-1} \beta_{\ell} (z-\m )^{\ell}$ for some $\beta_{\ell} \in \kk$. For $\varphi$ to be well-defined, we must have $\varphi(1) = \beta (z-\m )^{j-1}$ for some $\beta \in \kk$, and so $\Hom_{\Gr A}(M\ma, M\maj) = \kk$.

For uniqueness, we show that $\Ext^1_{\Gr A}(M_{\m}^{-(j-1)}, M\ma) = \kk$. We introduce temporary notation and let $I = (z-\m)^{j-1}A + xA$ be the ideal defining $M_{\m}^{-(j-1)}$. Applying the functor $\Hom_{\Gr A}(-, M\ma)$ to the exact sequence
\[ 0 \to I \to A \to M_{\m}^{-(j-1)} \to 0
\]
yields the exact sequence
\begin{align*} 0 &\to \Hom_{\Gr A}(M_{\m}^{-(j-1)}, M\ma) \to \Hom_{\Gr A}(A, M\ma) \\
&\to \Hom_{\Gr A}(I, M\ma) \to \Ext^1_{\Gr A}(M_{\m}^{-(j-1)},M\ma) \to 0.
\end{align*}
Since $M\ma$ is one-dimensional in degree 0, and there is a natural projection $M_{\m}^{-(j-1)} \to M\ma$, the first two terms in this sequence are both one-dimensional over $\kk$. Since $M\maj$ is a nontrivial extension of $M\ma$ by $M\maj$, the last term is at least one-dimensional over $\kk$. Now any morphism $\phi: I \to M\ma$ is determined by $\phi((z-\m)^{j-1})$ and $\phi(x)$. Since $M\ma$ is zero in all positive degree, $\phi(x) = 0$ and so $\phi$ is determined by the image of $(z-\m)^{j-1}$. Since $M\ma$ is one-dimensional in degree 0, we have $\Hom_{\Gr A}(I, M\ma)$ is at most one-dimensional over $\kk$. By considering $\kk$-dimension, this implies that every term in the exact sequence is one-dimensional.

Finally, to see that $\mcMma$ has projective dimension one, we use a technique similar to \cite[Theorem 5]{bav}. Observe that there is an exact sequence
\begin{align*}
0 \ra (z-\m )^{n_{\m }} A + xA  \overset{\phi}{\ra} A \oplus A\s{-1} \overset{\psi}{\ra} \left[\sigma^{-1}\left((z- \m )^{n_{\m }}\right)A + yA \right] \s{-1}\ra 0
\end{align*}
where $\phi(w) = (w, [\sigma^{-1}\left((z - \m )^{n_{\m }}\right) ]\inv y w)$ and $\psi(u,v) = yu - \sigma^{-1}\left((z-\m )^{n_{\m }}\right) v$. Then $\mcMma$ and $\mcMpa$ will have projective dimension one if this exact sequence splits; i.e., if there is a splitting map $\nu: A \oplus A \sra (z-\m )^{n_{\m }} A + xA$ such that $\nu \circ \phi= \Id$. This is equivalent to the existence of $a, b \in (z-\m )^{n_{\m }} A + xA$ so that $ 1 = a + b [\sigma^{-1}\left((z - \m )^{n_{\m }}\right) ]\inv y$ (if so, we can define the splitting map $\nu(u,v) = au + bv$). But this is true since $\gcd \left( (z-\m )^{n_{\m }}, f/(z-\m )^{n_{\m }} \right) = 1$ in $R$.
\end{proof}

The indecomposable modules $\mcMma$ and $\mcMpa$ will play an important role in constructing autoequivalences of $\gr A$ in section~\ref{sec:pic}.
The proof of the following lemma is the same as that of \cite[Lemma 3.4]{wonpic}.

\begin{lemma} \label{lem.support} Let $n \in \ZZ$. Then for each $\m  \in \Z $ and each $j = 1 ,\dots, n_{\m }$, 
\[
\ann_R (M\maj\s{n})_i=\ann_R (M\paj\s{n})_i=\sigma^{-n}((z-\m)^j)R,\quad\mathrm{for}\;i\leq -n,
\]
and $M\maj\s{n}=M\paj\s{n}=0$ for $i>-n$.
 As graded left $R$-modules, we have
\[ \frac{A}{(z-\m )A}\s{n} \cong \frac{A}{\sigma^{-1}(z-\m ) A}\s{n-1} \cong \bigoplus_{i \in \ZZ} \frac{R}{\sigma^{-n}(z-\m ) R}.
\]
\end{lemma}

It is easy to see that as a left $R$-module, $M_\lambda \cong \bigoplus_{j\in \ZZ} R/(z-\lambda)R$. This fact, when combined with the previous lemma and Lemma~\ref{lem.simples}, means that any finite-length graded $A$-module is supported at finitely many $\kk$-points of $\Spec R$ when considered as a left $R$-module. 

\begin{definition} If $M$ is a graded $A$-module of finite length, define the \emph{support of $M$}, $\Supp M$, to be the support of $M$ as a left $R$-module. If $\Supp M \subset \{\sigma_{\kk}^{i}(\m ) \mid i \in \ZZ \}$ for some $\alpha \in \kk$, we say that $M$ is \emph{supported along the $\sigma_\kk$-orbit of $\m $}. 
\end{definition}

By Lemmas~\ref{lem.simples} and~\ref{lem.support}, for any $\lambda \in \kk \setminus \{ \sigma_{\kk}^{i}(\m ) \mid \m \in \Z , i \in \ZZ \}$, the simple module $M_\lambda$ is the unique simple supported at $\lambda$ and for each $\m \in \Z$ and $n \in \ZZ$, $M\ma \s{n}$ and $M\pa\s{n}$ are the unique simple modules supported at $\sigma_{\kk}^{-n}(\m )$.

\subsection{Rank one projective modules}
\label{sec:ideals}
In this section, we seek to understand the rank one graded projective right $A$-modules. We begin by studying the graded submodules of $\QgrA$.

Let $I$ be a finitely generated graded right $A$-submodule of $\QgrA$. Since $R$ is a PID, for each $i \in \ZZ$, there exists $a_i \in \kk(z)$ so that as a left $R$-module,
\[ I = \bigoplus_{i\in \ZZ} R a_i x^i.
\]
For every $i \in \ZZ$, multiplying $I_i$ on the right by $x$ shows that $R a_{i} \subseteq R a_{i+1}$, while multiplying $I_i$ on the right by $y$ shows that $R a_{i+1} \sigma^i(f) \subseteq R a_i$.

Define $c_i = a_{i}a_{i+1}\inv$. Since $R a_i \subseteq R a_{i+1}$, therefore $c_i \in R$, and since $R a_{i+1} \sigma^i(f) \subseteq R a_i$, we conclude that $c_i$ divides $\sigma^{i}(f)$ in $R$. By multiplying by an appropriate unit in $R$, we may assume that $c_i \in \kk[z]$ and that $c_i$ is monic. Hence, $c_i$ actually divides $\sigma^{i}(f)$ in $\kk[z]$.


\begin{definition}\label{wonferpic} For a finitely generated graded submodule $I = \bigoplus_{i\in \ZZ} Ra_ix^i$ of $\QgrA$, we call the sequence $\{c_i = a_{i}a_{i+1}\inv\}_{i \in \ZZ}$ described above the \emph{structure constants of $I$}.
\end{definition}

Many properties of $I$ can be deduced by examining its structure constants. The proofs of the next two lemmas are immediate generalizations of the proofs of \cite[Lemmas 3.8 and 3.9]{wonpic}.

\begin{lemma} \label{sciso} Let $I$ and $J$ be finitely generated graded submodules of $\QgrA$ with structure constants $\{c_i\}$ and $\{d_i\}$, respectively. Then $I \cong J$ as graded right $A$-modules if and only if $c_i = d_i$ for all $i \in \ZZ$.
\end{lemma}

\begin{lemma} \label{fgsc} Suppose $I = \bigoplus_{i\in \ZZ} R a_i x^i$ is a finitely generated graded right $A$-submodule of $\QgrA$ with structure constants $\{c_i\}$. Then for $n \gg 0$, $c_n = 1$ and $c_{-n} = \sigma^{-n}(f)$. Further, for any choice $\{c_i\}_{i \in \ZZ}$ satisfying
\begin{enumerate}
\item for each $n\in \ZZ$, $c_n \in \kk[z]$, $c_n \mid \sigma^n(f)$ and
\item for $n \gg 0$, $c_n = 1$ and $c_{-n} = \sigma^{-n}(f)$,
\end{enumerate}
there is a finitely generated module $I$ with structure constants $\{c_i\}$.
\end{lemma} 

The structure constants of $I$ also determine which of the indecomposable modules described in Lemma~\ref{lem.indec} are homomorphic images of $I$. 

\begin{lemma}\label{cssc} Let $I = \bigoplus_{i\in \ZZ} R a_i x^i$ be a finitely generated graded right $A$-submodule of $\QgrA$ with structure constants $\{c_i\}$. For $\m  \in \Z $ and $j = 1, \dots, n_{\m }$
\begin{enumerate} 
\item \label{mcase} there is a surjective graded right $R$-module homomorphism $I \sra M\maj\s{n}$ if and only if $\sigma^{-n}\left((z-\m )^{n_{\m } - j + 1}\right) \nmid c_{-n}$,
\item \label{pcase} there is a surjective graded right $R$-module homomorphism $I \sra M\paj\s{n}$ if and only if $\sigma^{-n}\left((z-\m )^{j}\right) \mid c_{-n}$.
\end{enumerate}
Moreover, these surjections are unique up to a scalar.
\end{lemma}
\begin{proof}
We prove statement \eqref{mcase}; the proof of statement \eqref{pcase} is similar.
Suppose there is a graded surjective homomorphism $I \sra M\maj\s{n}$. 
The kernel of this morphism is again a graded submodule of $\QgrA$, so we can write the kernel as $J = \bigoplus_{i \in \ZZ} R b_i x^i$ with each $b_i \in \kk(z)$. Let $\{ d_i\}$ denote the structure constants of $J$. 
Since, $M\maj\s{n}$ is nonzero in precisely degrees $i \leq -n$, it follows that $b_i = a_i$ for all $i > -n$. 
Further, by Lemma~\ref{lem.support}, for all $i \leq -n$, $(M\maj\s{n})_i \cong R/\sigma^{-n}((z-\m )^j)R$ as left $R$-modules.
Therefore, for all $i \leq -n$, $Rb_i \supseteq R\sigma^{-n}((z-\m )^j)a_i$. We also notice that $\dim_\kk Ra_i/Rb_i=\dim_\kk R/R\sigma^{-n}((z-\m)^j)=j$, since $Rb_i\subseteq Ra_i$ and both $a_i$ and $b_i$ are monic, we deduce $b_i=\sigma^{-n}((z-\alpha)^j)a_i$.
Computing structure constants, we see that when $i \neq -n$, $c_i = d_i$, while $d_{-n} = \sigma^{-n}((z - \m )^j )c_{-n}$. 
Since by the discussion before Definition~\ref{wonferpic}, $d_{-n} \mid \sigma^{-n}(f)$, we conclude that
\[\sigma^{-n}\left((z-\m )^{n_{\m } - j + 1}\right) \nmid c_{-n},\]
otherwise $(z-\alpha)^{n_\m+1}$ would divide $f$.

Conversely, suppose that $\sigma^{-n}\left((z-\m )^{n_{\m } - j +1}\right) \nmid c_{-n}$. Then
we can construct a finitely generated graded right $A$-submodule of $\QgrA$ by setting $J = \bigoplus_{i \in \ZZ} R b_i x^i \subseteq I$ and $b_i = a_i$ for all $i > -n$ 
and $b_i = \sigma^{-n}((z - \m )^j )a_i$ for all $i \leq -n$. 
By multiplying by a scalar, we may assume that $b_i$ is monic for all $i$. 
For all $i \in \ZZ$, define $d_i = b_{i}b_{i+1}\inv$. 
Observe that for $i \neq -n$, $d_i = c_i$, while $d_{-n} = \sigma^{-n}((z-\m )^j )c_{-n}$. 
The $\{d_i\}$ are the structure constants of $J$, which is finitely generated by Lemma~\ref{fgsc}. 

Now by \cite[Theorem 2]{bavkdim}, $I/J$ has finite length. 
Further, $I/J$ is a graded module such that $(I/J)_i = 0$ for each $i > -n$ and $\dim_{\kk}(I/J)_i = j$ for each $i \leq -n$, since in this case $Ra_i/Rb_i\cong R/\sigma^{-n}((z-\alpha)^j)$.
Hence, by Lemma~\ref{lem.simples} and by looking at the degrees in which the graded simple modules are nonzero, we deduce that the composition factors of $I/J$ are $j$ copies of $M\ma\s{n}$.
Finally, $I/J$ has a unique submodule which is isomorphic to $M\ma\s{n}$, 
corresponding to the module $J \subseteq \bigoplus_{i\in\ZZ}R b_i'x^i \subseteq I$ where $b'_i = a_i$ for all $i > -n$ 
and $b'_i = \sigma^{-n}((z - \m )^{j-1} )a_i$ for all $i \leq -n$. 
Therefore, $I/J$ is indecomposable and so by Lemma~\ref{lem.indec}, $I/J \cong M\maj\s{n}$.
We note that we have constructed the unique kernel of a morphism $I \sra M\maj$ and so the surjection is unique up to a scalar.
\end{proof}

We now focus on the rank one projective modules. Let $P$ be a finitely generated graded projective right $A$-module of rank one. Since $P$ embeds in $Q(A)$ and is graded it follows that $P$ embeds in $\QgrA$, therefore $P$ has a sequence of structure constants. We will be able to detect the projectivity of $P$ from its structure constants.

\begin{lemma}\label{lem.lifts} Let $P$ be a graded projective right $A$-module. Let $\m \in \Z$ and $j = 1, \dots, n_{\m }$. Any graded surjection $P \sra M\pma \s{n}$ lifts to a graded surjection $P \sra M \pmaj \s{n}$ which is unique up to a scalar. 
\end{lemma}
\begin{proof} It suffices to prove that for any $j = 2 ,\dots , n_{\m }$ any graded surjection $P \sra M_{\m }^{-(j-1)}$ lifts to a graded surjection $P \sra M\maj$.

Let $\pi$ be the projection $\pi: M\maj \sra M_{\m }^{-(j-1)}$. Since $P$ is projective, any graded surjection $P \sra M_{\m }^{-(j-1)}$ lifts to a graded morphism $g: P \sra M\maj$. Now since the kernel of $\pi$ is generated by $(z-\m )^{j-1}$, any preimage of $1$ under $\pi$ has the form $p = 1 + (z-\m )^{j-1}a $ for some $a\in A_0$. Observe that $p(1-(z-\m )^{j-1} a ) = 1$ in $M\maj$, so any preimage of $1$ generates $M\maj$ and hence $g$ is an surjection. By Lemma~\ref{cssc}, $g$ is unique up to a scalar. In the other case, the proof is analogous.
\end{proof}

The following lemma generalizes \cite[Lemma 3.32]{wonpic}.

\begin{lemma}\label{mprojsc}Let $P$ be a graded submodule of $\QgrA$ with structure constants $\{c_i\}$. Then $P$ is projective if and only if for each $n \in \ZZ$, $c_{n} = \prod_{\m \in I}\sigma^{n}\left((z- \m )^{n_{\m }}\right)$ for some (possibly empty) subset $I \subseteq \Z$.
\end{lemma}
\begin{proof} 
We first assume that for each $n\in\mathbb{Z}$, $c_n=\prod_{\m\in I}\sigma^n((z-\m)^{n_\m})$ for some $I\subseteq \Z$ and prove that $P$ must be projective. As in  \cite[Lemma 3.32]{wonpic}, we prove the projectivity of $P$ by constructing a finitely generated projective module with the same structure constants as $P$, the claim would then follow by Lemma~\ref{sciso}.

Note that by Lemma~\ref{fgsc}, there exists $N_1 \in \mathbb{Z}$ such that for all $n \geq N_1$, $c_n = 1$ and for all $n \leq -N_1$, $c_{n} = \sigma^{n}(f)$. An elementary computation shows that if $P_0=A\langle N_1\rangle$ and the structure constants of $P_0$ are denoted by $\{d_i\}$, then $d_n=1$ for $n\geq N_1$ and $d_n=\sigma^n(f)$ for $n< N_1$. If $c_n=d_n$ for all $n$ then we are done. Otherwise let $i_0$ be the largest integer such that $c_{i_0}\neq d_{i_0}$. By hypothesis, there is some $I \subseteq \Z$ such that $c_{i_0} = \prod_{\m \in I }\sigma^{i_0}\left((z- \m )^{n_{\m }}\right)$ and by construction, $d_{i_0} = \sigma^{i_0}(f)$. Let $J = \Z \setminus I$.

Since $\sigma^{i_0}((z-\alpha)^{n_\alpha})\mid \sigma^{i_0}(f)=d_{i_0}$ then, by Lemma~\ref{cssc}, there is a graded surjection $\pi_{\m}:P_0 \sra \mcMpa\s{-i_0}$ for each $\m \in \Z$. Let $P_1$ be the kernel of the map $\bigoplus_{\m\in J}\pi_{\m}: P_0 \sra \bigoplus_{\m \in J} \mcMpa\s{-i_0}$. By Lemma~\ref{lem.indec}, the modules $ \mcMpa\s{-i_0}$ have projective dimension one, therefore the projectivity of $P_0$ implies the projectivity of $P_1$. Following the same strategy used in the first paragraph of the proof of Lemma~\ref{cssc} we see that $P_1$ is a submodule of $\QgrA$ and the structure constants of $P_0$ and $P_1$ are equal except in degree $i_0$, where $P_1$ has structure constant $c_{i_0}$. We continue this process for the finitely many indices where $c_i$ differs from $d_i$ until we reach a projective module which has the same structure constants as $P$.

Conversely suppose that $P$ is projective. Using Lemma~\ref{cssc} and the above lemma, it follows that if $\sigma^{n}(z-\m ) \mid c_n$, then $\sigma^n((z-\m )^{n_{\m }}) \mid c_n$. Since by the paragraph before Definition~\ref{wonferpic} we know that $c_n\mid \sigma^n(f)$, the assertion follows.
\end{proof}

Since projective graded submodules of $\QgrA$ have rank one, as a corollary, we see that a graded rank one projective module has a unique simple factor supported at $\sigma_{\kk}^n(\m )$ for each $\m \in Z$ and each $n \in \ZZ$.

\begin{corollary}\label{projXY}Let $P$ be a rank one graded projective $A$-module, let $\m \in \Z$, and let $n \in \ZZ$. Then $P$ surjects onto exactly one of $M\ma\s{-n}$ and $M\pa\s{-n}$.
\end{corollary}
\begin{proof}
By Lemma~\ref{mprojsc} each structure constant $c_{n}$ of $P$ is either relatively prime to $\sigma^{n}(( z- \m)^{n_\m})$ or else has a factor of $\sigma^{n}(( z- \m)^{n_\m})$. By Lemma~\ref{cssc}, in the first case, $P$ surjects onto $M\ma\s{-n}$ and not $M\pa\s{-n}$, and in the second case $P$ surjects onto $M\pa \s{-n}$ and not $M\ma\s{-n}$.
\end{proof}

\begin{corollary}\label{projfactors}A rank one graded projective $A$-module is determined up to isomorphism by its simple factors which are supported at the $\sigma_{\kk}$-orbits of the roots $\m $ of $f$.
\end{corollary}
\begin{proof} Let $P$ be a rank one graded projective $A$-module with structure constants $\{c_i\}$. By Lemma~\ref{cssc} and Corollary~\ref{projXY}, the simple factors of $P$ supported at the $\sigma_{\kk}$-orbits of the roots $\m $ of $f$ determine the $\{c_i\}$ and by Lemma~\ref{sciso} therefore determine $P$.
\end{proof}



We have shown that for a rank one projective $P$ for each $\m \in \Z$ and each $n \in \ZZ$, $P$ has a unique simple factor supported at $\sigma_{\kk}^{n}(\m )$---either $M\ma\s{-n}$ or $M\pa\s{-n}$. The next result show that if for each $n \in \ZZ$ and each $\m \in \Z$, we choose one of $M\ma\s{-n}$ and $M\pa\s{-n}$, as long as our choice is consistent with Lemma~\ref{fgsc}, there exists a projective module with the prescribed simple factors.

\begin{lemma} \label{anyproj}
For each $\m  \in \Z  $ and each $n \in \ZZ$ choose $S_{\m , n} \in \{M\ma\s{-n}, M\pa\s{-n}\}$ such that for $n \gg 0$, $S_{\m , n} = M\ma\s{-n}$ and $S_{\m , -n} = M\pa\s{n}$. Then there exists a finitely generated rank one graded projective $P$ such that for all $n \in \ZZ$, the unique simple factor of $P$ supported at $\sigma_{\kk}^{n}(\m )$ is $S_{\m , n}$.
\end{lemma}
\begin{proof}We construct $P$ via its structure constants $\{c_i\}$ as follows. For each $\m \in \Z$ and $i \in \ZZ$, let 
\[ d_{\m , i} = \begin{cases}1 & \mbox { if } S_{\m , i} = M\ma\s{-i}\\
\sigma^{-i}\left((z- \m )^{n_{\m }}\right) & \mbox{ if } S_{\m , i} = M \pa\s{-i} 
\end{cases}
\]
and let $c_i = \prod_{\m \in \Z} d_{\m ,i}$. By Lemmas~\ref{fgsc} and~\ref{mprojsc}, $P$ is a finitely generated rank one graded projective module.
\end{proof}

\subsection{Morphisms between rank one projectives}

In \cite[\S 3.4.2]{wonpic}, the morphisms between the rank one projectives of $\gr A$ were described in the case that $A = \kk[z](\sigma, f)$ for $\sigma(z) = z+1$ and $f \in \kk[z]$ quadratic. The hypotheses used therein were that $\kk[z]$ is a PID and $A$ is $1$-critical, i.e., $A$ has Krull dimension $1$ and for any proper submodule $B \subseteq A$, $A/B$ has Krull dimension $0$. Since both $\kk[z]$ and $\kk[z, z\inv]$ are PIDs and all of the simple GWAs in this paper are 1-critical, the results generalize immediately. We briefly review the pertinent results and definitions from \cite{wonpic}.

\begin{definition}Given a graded rank one projective module $P$ with structure constants $\{c_i\}$, the \emph{canonical representation of $P$} is the module 
\[P' = \bigoplus_{i \in \ZZ} R p_i x^i \subseteq \QgrA
\] where
$p_i = \prod_{j \geq i} c_j$. We note that $P' \cong P$ and call $P'$ a \emph{canonical rank one graded projective module}. Every rank one graded projective right $A$-module is isomorphic to a unique canonical one in this way.
\end{definition}

Let $P$ and $Q$ be finitely generated right $A$-modules. An $A$-module homomorphism $f: P \sra Q$ is called a \emph{maximal embedding} if there does not exist an $A$-module homomorphism $g: P \sra Q$ such that $f(P) \subsetneq g(P)$.
In \cite{wonpic}, it was proved that if $P$ and $Q$ are finitely generated graded rank one projectives, then there exists a maximal embedding $P \sra Q$ which is unique up to a scalar. Further, if $P$ and $Q$ are \emph{canonical} rank one graded projectives, this maximal embedding can be computed explicitly.

\begin{proposition}[{\cite[Proposition 3.38]{wonpic}}] \label{projhom} Let $P$ and $Q$ be finitely generated graded rank one projective $A$-modules embedded in $\QgrA$. Then every homomorphism $P \sra Q$ is given by left multiplication by some element of $\kk(z)$ and as a left $R$-module, $\Hom_{\gr A}(P,Q)$ is free of rank one.
\end{proposition}

\begin{lemma}[{\cite[Lemma 3.40 and Corollary 3.41]{wonpic}}] \label{projmaxemb} Let $P$ and $Q$ be rank one graded projective $A$-modules with structure constants $\{c_i\}$ and $\{d_i\}$, respectively. As above, write 
\[ P = \bigoplus_{i \in \ZZ} R p_i  x^i = \bigoplus_{i \in \ZZ} \left( R \prod_{j \geq i} c_j \right)  x^i \mbox{ and } Q = \bigoplus_{i \in \ZZ} Rq_i x^i = \bigoplus_{i \in \ZZ} \left( R \prod_{j \geq i} d_j  \right) x^i.
\]
Then the maximal embedding $P \sra Q$ is given by multiplication by 
\[
\theta_{P,Q} = \lcm_{i \in \ZZ} \left( \frac{q_i}{\gcd(p_i, q_i)}\right)= \lcm_{i \in \ZZ} \left( \frac{\prod_{j \geq i}d_j}{\gcd\left( \prod_{j \geq i}c_j, \prod_{j \geq i} d_j\right)} \right)
\]
where $\lcm$ is the unique monic least common multiple in $R$. There exists some $N\in\ZZ$ such that
\[ \theta_{P,Q} = \frac{q_N}{\gcd(p_N,q_N)}.
\]
\end{lemma}

Finally, we give necessary and sufficient conditions for a set of projective objects in $\gr A$ to generate the category. This is the natural generalization of the conditions in \cite{wonpic} and follows from the same proof.

\begin{lemma}[{\cite[Proposition 3.43]{wonpic}}]\label{projgen} A set of rank one graded projective $A$-modules $\mcP = \{P_i\}_{i \in I}$ generates $\gr A$ if and only if for every $M_\m^\pm\langle n\rangle$ with $\m\in\Z$ and $n\in\ZZ$, there exists a graded surjection to $M_\m^\pm\langle n\rangle$ from a direct sum of modules in $\mcP$.
\end{lemma}

\subsection{\texorpdfstring{Involutions of $\gr A$}{Involutions of gr-A}}
\label{sec:pic}

We now construct autoequivalences of $\gr A$ which are analogous to those constructed in \cite[Proposition 5.7]{sierra} and \cite[Propositions 5.13 and 5.14]{wonpic}. Recall that $\mcS$ denotes the shift functor on $\gr A$.

\begin{proposition}\label{prop.iotas}
Let $A = R(\sigma,f)$ and let $\Z$ be the set of roots of $f$. For any $\m \in \Z$ and any $j \in \ZZ$, there is an autoequivalence $\iota^{\m }_{j}$ of $\gr A$ such that $\iota^{\m}_{j}(M\pma\s{j}) \cong M\mpa\s{j}$ and $\iota^{\m}_{j}(S) \cong S$ for all other graded simple $A$-modules $S$. For any $j, k \in \ZZ$, $\mcS_A^j \iota^{\m}_{k} \cong \iota^{\m}_{j+k}\mcS_A^j$, and $(\iota^{\m}_{j})^2 \cong \Id_{\gr A}$.
\end{proposition}

\begin{proof}
For each $\m \in \Z$, we will construct $\iota^{\m}_{0}$ and then define $\iota^{\m}_{j}$ as $\mcS_A^j \iota^{\m}_{ 0} \mcS_A^{-j}$. 

Let $\sfR$ denote the full subcategory of $\gr A$ consisting of the canonical rank one graded projective modules. We first define $\iota^{\m}_{0}$ on $\sfR$. Consider the set
\[ D_\m = \left\{ M\mai, M\pai \st i = 1, \dots, n_{\m }\right\}\]
and let $\sfD_\m$ denote the full subcategory of $\gr A$ whose objects are the elements of $D_\m$. 

Let $P$ be an object of $\sfR$. By Lemma~\ref{lem.lifts} and Corollary~\ref{projXY}, there is an unique (up to a scalar) surjection from $P$ onto $\mcMma$ or $\mcMpa$, let $N$ be the kernel of this surjection (which does not depend on the particular choice of the surjection), so $P/N\in \sfD_\m$. As $\sfD_\m$ is closed under subobjects, the functor $\sfR\sra\gr A$ that maps an object $P$ to this unique kernel and acts on morphisms by restriction is a well-defined additive functor by \cite[Lemma 4.4]{wonpic}.  By  \cite[Lemma 4.3]{wonpic}, this then extends to an additive functor defined over the full subcategory of direct sums of rank one projective modules. By  \cite[Lemma 4.2]{wonpic} this functor further extends to an additive functor $\iota^{\m}_{ 0}: \gr A \sra \gr A$. 

We now show that $\iota^{\m}_{ 0}$ has the claimed properties. We begin by describing the action on structure constants. Suppose $P \in \sfR$ has structure constants $\{c_i\}$ and denote the structure constants of $\iota^{\m }_0 P$ by $\{d_i\}$. By Lemma~\ref{cssc}, $P$ surjects onto exactly one of $\mcMma$ and $\mcMpa$. If $P$ surjects onto $\mcMma$, then a similar strategy to the one used in the first paragraph of the proof of Lemma~\ref{cssc} can be used to show that
\[ (\iota^{\m}_{0}P)_n = \begin{cases} (z- \m )^{n_{\m }} P_n & \mbox{ if } n \leq 0 \\
P_n & \mbox{ if } n > 0.
\end{cases}
\]
Similarly, if $P$ surjects onto $\mcMpa$ then
\[ (\iota^{\m}_{0}P)_n = \begin{cases} P_n & \mbox{ if } n \leq 0 \\
(z-\m )^{n_{\m }} P_n & \mbox{ if } n > 0.
\end{cases}
\]
In the first case, $d_0 = c_0 (z-\m )^{n_{\m }}$ while in the second case, $d_0 = c_0/(z-\m  )^{n_{\m }}$. Therefore by Lemma~\ref{cssc} if $P$ surjects onto $M^{\pm n_\m}_\m$ then $\iota^{\m}_0P$ surjects onto $M^{\mp n_\m}_\m$.
Hence $(\iota^{\m}_{ 0})^2 P = (z-\m )^{n_{\m }}P$, which is isomorphic to $P$. 

Let $P'$ be another object in $\sfR$. Consider map induced by $(\iota_0^\m)^2$
\[
 \Hom_{\gr A}(P, P')\rightarrow\Hom_{\gr A}((z-\m )^{n_{\m }} P, (z-\m )^{n_{\m }} P') ,
\]
given by $g\mapsto g|_{(z-\m)^{n_\m}P}$. By Proposition~\ref{projhom}, every element of $ \Hom_{\gr A}(P, P')$ and of $\Hom_{\gr A}((z-\m )^{n_{\m }} P, (z-\m )^{n_{\m }} P')$ is given by left multiplication by some element of $\kk(z)$. Therefore the map in the above display must be an isomorphism. If $Q$ is a finite direct sum of rank one graded projective modules then, by the additivity of $\iota^\m_0$, $(\iota^{\m}_{ 0})^2$ is given by multiplication by $(z-\m )^{n_{\m }}$ in each component of $Q$. Hence, $(\iota^{\m}_{ 0})^2$ is naturally isomorphic to the identity functor on the full subcategory of finite direct sums of rank one projectives. Therefore, by \cite[Lemma 4.2]{wonpic}, $\iota^{\m}_{ 0}$ is an autoequivalence of $\gr A$ which is its own quasi-inverse.

If $\iota^\m_j$ is defined as $\iota^{\m}_{j} = \mcS_A^j \iota^{\m}_{ 0} \mcS_A^{-j}$ then
\begin{align*}
\mcS^j_A\iota^\m_k&=\mcS^j_A\mcS^k_A\iota^\m_0\mcS^{-l}_A\\
&=\mcS^{j+k}_A\iota^\m_0\mcS^{-j-k}_A\mcS^j_A\\
&=\iota^\m_{j+k}\mcS^j_A.
\end{align*}

Because for any $P \in \sfR$ the structure constants for $\iota^{\m}_{ 0} P$ differ from those of $P$ only in degree $0$, where they differ only by a factor of $(z-\m )^{n_{\m }}$, by Lemma~\ref{cssc}, $\iota^{\m}_{ 0} P$ and $P$ have the same simple factors supported along $\{\sigma^{i}_\kk(\m ) \mid i \in \mathbb{Z}, \m \in \Z\}$ except if $P$ has a factor of $M\ma\s{j}$ then $\iota^{\m}_{ 0} P$ has a factor of $M\pa\s{j}$ and vice versa. Let $S$ be a simple module not of the form $M\pma\s{j}$, i.e. $S$ is $M_\beta^\pm\s{j}$ for some $\beta\in\Z\backslash\{\m\}$, or $S$ is $M_\lambda\s{j}$ for some $\lambda$. By looking at the way $\iota^\m_0$ is defined on $\gr A$, in \cite[Lemmas 4.2 and 4.3]{wonpic}, it follows that $\iota^\m_0(S)=S$.
\end{proof}

\begin{remark}
\label{rem.iotas}
Since we defined $\iota^{\m}_{j} = \mcS^j_A \iota^{\m}_{0} \mcS^{-j}_A$, we can construct $\iota^{\m}_{j}$ by adjusting the construction in the previous proof by shifting all of the modules in $D$ by $j$. It follows that for a rank one graded projective module $P$, $(\iota^{\m }_{j})^2 P = \sigma^{-j}\left((z-\m )^{n_\m }\right) P$. 
\end{remark}

\begin{lemma} \label{lem.pic}
For any $\m, \beta \in \Z$ and any $j, k \in \ZZ$, $\iota^{\m}_{j} \iota^{\beta}_{k} = \iota^{\beta}_{k} \iota^{\m}_{j}$ and so the autoequivalences $\{ \iota^{\m}_{j} \mid \m \in \Z, j\in \ZZ \}$ generate a subgroup of $\Pic(\gr A)$ isomorphic to 
\[\bigoplus_{\m \in \Z} \bigoplus_{j \in \ZZ} \ZZ/2\ZZ.\]
\end{lemma}
\begin{proof}Since, for each $\m \in \Z$ and $j \in \ZZ$, $(\iota^{\m}_{j})^2 \cong \Id_{\gr A}$, each $\iota^{\m}_{j}$ generates a subgroup of $\Pic(\gr A)$ isomorphic to $\ZZ/2\ZZ$. It remains to show that for any $\beta \in \Z$ and $k \in \ZZ$ that $\iota^{\m}_{j}$ and $\iota^{\beta}_{k}$ commute. Tracing through the construction in the proof of Proposition~\ref{prop.iotas} shows that $(\iota^{\beta}_{k})\inv (\iota^{\m}_{j})\inv \iota^{\beta}_{k} \iota^{\m}_{j} \cong \Id_{\gr A}$.
\end{proof}

\section{Quotient stacks as noncommutative schemes of GWAs}
\label{sec.comm}

Let $A = R(\sigma,f)$ and fix a labeling of the distinct roots of $f$ 
\[ \Z = \{\m_1, \dots, \m_r\}
\]
where $\m_i $ has multiplicity $n_i$ in $f$.
The main aim of this section is to use the autoequivalences constructed in the previous section to construct a $\Gamma$-graded commutative ring $B$ whose category of $\Gamma$-graded modules is equivalent to $\Gr A$. We will then study properties of the ring $B$.

We identify $\bigoplus_{i \in \ZZ} \ZZ/2\ZZ$ with the group $\ZZ_\fin$ of finite subsets of the integers. 
The operation on $\ZZ_\fin$ is given by exclusive or, denoted $\oplus$. 
For convenience, we use simply $j$ to denote the singleton set $\{j \} \in \ZZ_\fin$. Throughout, we write the group of autoequivalences described in Lemma~\ref{lem.pic} as
\[ \Gamma = \bigoplus_{i = 1}^r \ZZ_\fin.
\]
We note that $\ZZ_\fin \oplus \ZZ_\fin \cong \ZZ_\fin$ and therefore $\Gamma \cong \ZZ_\fin$, but it will be convenient to index this group by the roots of $f$. We use the notation $[1,r]$ to refer to the set $\{1,2, \dots, r\}$.

For $i \in [1,r]$ and $j \in \ZZ$, let $\ule_{i,j} = \{\emptyset, \dots, j, \dots, \emptyset\} \in \Gamma$, where $j$ is in the $i^\th$ component.
We denote the identity $(\emptyset, \dots, \emptyset)$ of $\Gamma$ by $\emptyset$.

For each $\ulJ = (J_1, J_2,  \dots, J_r) \in \Gamma$, let
\[ \iota_{\ulJ} = \prod_{i = 1}^{r} \prod_{j \in J_i}  \iota^{\m_i}_{j}.
\]
Since, for each $j \in \ZZ$, the autoequivalence $\iota^{\m_i}_{ j}$ is its own quasi-inverse, therefore $\iota_{\ulJ}$ is also its own quasi-inverse.

\begin{lemma}\label{lem.iotasgen} 
The set $\{ \iota_{\ulJ} A \mid \ulJ \in \Gamma\}$ generates $\Gr A$.
\end{lemma}
\begin{proof} We remark that by the construction in Proposition~\ref{prop.iotas},  if $P$ is a rank one projective, then $\iota_{\ulJ} P$ is a graded submodule of $\QgrA$ with the same structure constants as $P$ except in certain degrees, where the structure consants of $\iota_{\ulJ} P$ differ from those of $P$ by a factor of $\prod_{i \in I} (z- \m _i)^{n_{i}}$ for some $I \subseteq \{1, 2, \dots r\}$. Hence, by Lemma~\ref{mprojsc}, $\iota_{\ulJ} P$ is also a rank one graded projective module. By Lemma~\ref{projgen}, it is enough to show that for every $i = 1, \dots, r$ and every $n \in \ZZ$, there is a surjection from some $\iota_{\ulJ}A$ to $M_{\alpha_i}^-\s{n}$ and similarly for $M_{\alpha_i}^+\s{n}$. By Corollary~\ref{projfactors}, $A$ surjects onto exactly one of $M_{\alpha_i}^- \s{n}$ and $M_{\alpha_i}^+ \s{n}$. Then $\iota_{\ule_{i,n}}A$ surjects onto the other, so $\{ \iota_{\ulJ} A \mid \ulJ \in \Gamma\}$ generates $\gr A$ and hence generates $\Gr A$.
\end{proof}

For $\ulJ = (J_1, \dots, J_r), \ulK = (K_1, \dots, K_r) \in \Gamma$, we define
\[ \ulJ \cap \ulK = (J_1 \cap K_1, \dots, J_r \cap K_r) \quad\text{and}\quad  \ulJ \cup \ulK = (J_1 \cup K_1, \dots, J_r \cup K_r),
\]
in the natural way, and similarly define
\[ \ulJ \oplus \ulK = (J_1 \oplus K_1, \dots, J_r \oplus K_r).\]
For each $i = 1, \dots, r$ and each $j \in \ZZ$, we define the polynomial $h^{\m_i}_{j} = \sigma^{-j}\left( (z- \m_i)^{n_i} \right)$. If $\ulJ \in \Gamma$, let
\[h_{\ulJ} = \prod_{i = 1}^{r} \prod_{j \in J_i} h^{\m_i}_{ j}.
\]
For completeness, empty products are defined to be $1$.

Let $\ulI, \ulJ\in \Gamma$. Since, in $\Pic(\gr A)$, $\iota_{\ulI}\iota_{\ulJ} = \iota_{\ulI \oplus \ulJ}$, therefore $\iota_{\ulI}\iota_{\ulJ}$ is naturally isomorphic to $\iota_{\ulI \oplus \ulJ}$ and in particular, $\iota_{\ulI}\iota_{\ulJ}A \cong \iota_{\ulI \oplus \ulJ}A$. We now explicitly describe this isomorphism.
By Remark~\ref{rem.iotas}, $\iota_{\ulJ}^2 A = h_{\ulJ} A$.
Denote by $\tau_{\ulJ}$ the isomorphism $\iota_{\ulJ}^2 A \sra A$ given by left multiplication by $h_{\ulJ}^{-1}$. 
Define $\Theta_{\ulI,\ulJ}: \iota_{\ulI} \iota_{\ulJ} A = \iota_{\ulI \oplus \ulJ} \iota_{\ulI \cap \ulJ}^2 A \sra \iota_{\ulI \oplus \ulJ} A$ by
\[\Theta_{\ulI,\ulJ} = \iota_{\ulI \oplus \ulJ}(\tau_{\ulI \cap \ulJ}) = \Theta_{\ulJ,\ulI}.
\]

By a proof that is identical to that of \cite[Lemma 4.1]{woncomm}, we have that for any $\ulI, \ulJ, \ulK \in \Gamma$ and any $\varphi \in \Hom_{\gr A}(A, \iota_{\ulI} A)$,
\[ \Theta_{\ulK, \ulI \oplus \ulJ} \circ \iota_{\ulK}(\Theta_{\ulJ, \ulI}) \circ \iota_{\ulK} \iota_{\ulJ} (\varphi) = \Theta_{\ulJ \oplus \ulK, \ulI} \circ \iota_{\ulJ \oplus \ulK}(\varphi) \circ \Theta_{\ulK, \ulJ}.
\]
Therefore, by \cite[Proposition 3.4]{woncomm} we can define the $\Gamma$-graded ring
\begin{equation} \label{Bdef} B(R, \sigma, f) = B = \bigoplus_{\ulJ \in \Gamma} B_{\ulJ} = \bigoplus_{\ulJ \in \Gamma} \Hom_{\gr A}(A, \iota_{\ulJ}A).
\end{equation}
For $a \in B_{\ulI}$ and $b \in B_{\ulJ}$ multiplication in $B$ is given by $a \cdot b = \iota_{\ulI \oplus \ulJ}(\tau_{\ulI \cap \ulJ}) \circ \iota_{\ulJ}(a) \circ b$.

In the remainder of this section, we assume that $B = B(R,\sigma,f)$ is defined as in \eqref{Bdef}.

\begin{theorem}\label{Bequivalent} There is an equivalence of categories 
\[\Gr A \equiv \Gr (B, \Gamma).\]
\end{theorem}
\begin{proof} This is an immediate corollary of Lemma~\ref{lem.iotasgen} and \cite[Proposition 3.6]{woncomm}.
\end{proof}

\begin{theorem}\label{Bpresentation} The ring $B$ is a commutative ring with presentation
\[ B \cong \frac{R\left[b_{i,j} \st i \in [1,r], j \in \ZZ \right]}{\left(b_{i,j}^2 = \sigma^{-j}((z - \m_i  )^{n_i}) \st i \in [1,r], j \in \ZZ \right)}.
\]
There is a $\Gamma$-grading on $B$, given by $R = B_{\emptyset}$ and $\deg b_{i,j} = \ule_{i,j}$.
\end{theorem}
\begin{proof} 
By the construction Proposition~\ref{prop.iotas}, for each $\m \in \Z$ and each $j \in \ZZ$, $\iota^{\m}_{j} A$ is the kernel of the unique graded surjection from $A$ to $\mcMma \s{j} \oplus \mcMpa \s{j}$. Both $\mcMma \s{j}$ and $\mcMpa \s{j}$ have $\kk$-dimension $n_{\m }$ in all graded components in which they are nonzero. Therefore, $(\iota^{\m}_{ j} A)_0$ has codimension $n_{\m }$ in $A_0$. 
Further, $((\iota^{\m}_{ j})^2 A)_0 = (h^{\m}_{ j} A)_0$ and since $((\iota^{\m}_{ j})^2 A)_0 \subseteq (\iota^{\m}_{ j} A)_0 \subseteq A$, by comparing codimensions we conclude that 
\[ (\iota^{\m}_{ j} A)_0 = (h^{\m}_{ j}  A)_0 = h^{\m}_{ j}R.
\]

Thus, for $\ulJ \in \Gamma$, $(\iota_{\ulJ} A)_0 = h_{\ulJ} R$. We use the isomorphism $\varphi_{\ulJ}: (\iota_{\ulJ} A)_0 \sra \Hom_{\gr A}(A, \iota_\ulJ A)$, to identify $(\iota_{\ulJ} A)_0 = h_{\ulJ} R$ with $B_{\ulJ}$. Let $b_{\ulJ} = \varphi_{\ulJ}(h_{\ulJ})$.
Since the autoequivalences $\{ \iota_{\ulJ} \mid \ulJ \in \Gamma\}$ act on morphisms by restriction, using the definition of multiplication in $B$, one can check that for $\ulI, \ulJ \in \Gamma$,
\[b_{\ulI} b_{\ulJ} = b_{\ulI \cap \ulJ}^2 b_{\ulI \oplus \ulJ}.
\]

For each $i \in [1,r]$ and $j \in \ZZ$, let
\[ b_{i,j} = b_{\ule_{i,j}}.
\]
By the above computation, the $b_{i,j}$ commute and $b_{\ulJ} = \prod_{i=1}^r \prod_{j \in J_i} b_{i,j}$.
Since $h_{\ulJ}$ freely generates $(\iota_{\ulJ}A)_0$ as an $R$-module, $b_{\ulJ}$ freely generates $B_{\ulJ}$ as a $B_{\emptyset} = R$-module. Hence, it is clear that the $b_{i,j}$ generate $B$ as an $R$-module. Again using the definition of multiplication in $B$, $b_{i,j}^2 = h^{\m_i}_j$.

Finally, the proof of \cite[Proposition 4.4]{woncomm} shows that the ideal 
\[\left(b_{i,j}^2 = \sigma^{-j}((z - \m_i  )^{n_i{ }}) \st i \in [1,r], j \in \ZZ \right)\]
contains all of the relations among the $b_{i,j}$ and hence $B$ has the claimed presentation.
\end{proof}

\begin{corollary}\label{cor.stacks} Let $A = \bigoplus_{i \in \ZZ} A_i$ be a simple $\ZZ$-graded domain of GK dimension 2 with $A_i \neq 0$ for all $i \in \ZZ$. Then there is a commutative $\ZZ_\fin$-graded ring $B$ so that
\[ \Gr A \equiv \Qcoh \left[ \frac{\Spec B}{\Spec \kk \ZZ_\fin} \right].
\]
\end{corollary}
\begin{proof} This follows from standard results \cite[Tag 06WS]{stacks-project}, along with Theorems~\ref{Bequivalent} and~\ref{Bpresentation}. The $\ZZ_\fin$-grading on $B$ gives an action of $\Spec \kk \ZZ_\fin$ on $\Spec B$. We can therefore take the quotient stack $\chi = \left[\frac{\Spec B}{\Spec \kk \ZZ_\fin} \right]$. The category of $\ZZ_\fin$-equivariant quasicoherent sheaves on $\Spec B$ is equivalent to the category $\Qcoh (\chi)$ of quasicoherent sheaves on $\chi$. Hence, $\Gr A \equiv \Gr (B, \ZZ_\fin) \equiv \Qcoh(\chi)$.
\end{proof}

Unfortunately, since the ring $\kk \ZZ_{\fin}$ is not locally of finite type, the quotient stack $\chi$ is not easy to study. We devote the remainder of the paper to the study of the ring $B$.

\begin{corollary}
Suppose that $A$ is a GWA satisfying Hypothesis~\ref{hypa}. Then there exists a quotient stack $\chi$ such that $\QGr A \equiv \Qcoh \chi$.
\end{corollary}
\begin{proof}This follows immediately from Corollary~\ref{cor.qgr} and Corollary~\ref{cor.stacks}.
\end{proof}

\begin{theorem}\label{NonNoeth} The ring $B$ is non-noetherian.
\end{theorem}
\begin{proof}
It is enough to prove that the quotient of $B$ by the ideal generated by $z-1$ is non-noetherian. This quotient is of the form
\[
\frac{B}{(z-1)}=\frac{\kk[b_{i,j} \mid i \in [1,r], j \in \ZZ]}{\left(b_{i,j}^2-\gamma_{i,j} \st i \in [1,r], j \in \ZZ\right)},
\]
where the $\gamma_{i,j}$ are elements of $\kk$ depending on $\m _i$, $j$, $n_i$, and on the defining automorphism $\sigma$ of $A$. In particular, if $R = \kk[z]$ and $\sigma(z) = z-1$ then $\gamma_{i,j} = (1-j-\m _i)^{n_i}$ while if $R = \kk[z,z\inv]$ and $\sigma(z) = \xi z$, then $\gamma_{i,j} = (\xi^{-j} - \m _i)^{n_i}$. In either case, we claim that the ideals
\[
I_n=(b_{1,j}+\sqrt{\gamma_{1,j}}\mid 0\leq j<n)
\]
form a non-stabilizing ascending chain of ideals in $B/(z-1)$. It suffices to prove that $b_{1,n}+\sqrt{\gamma_{1,n}}$ is not in $I_n$. Indeed if we further quotient $B/(z-1)$ by $I_n$ we get a ring of the form
\[
\frac{B}{(z-1)+I_n}=\frac{\kk[b_{i,j} \mid i \in [1,r], j \in \ZZ ]}{\begin{pmatrix}b_{1,j}+\sqrt{\gamma_{1,j}}&\mid &0\leq j<n\\
b_{1,j}^2-\gamma_{1,j}&\mid& j\geq n\\
b_{i,j}^2-\gamma_{i,j}&\mid &i \in [2,r] \end{pmatrix}}.
\]
We prove that $b_{1,n}+\sqrt{\gamma_{1,n}}$ is not zero in this quotient by  proving that it does not belong to the ideal
\[
J_n=\begin{pmatrix}b_{1,j}+\sqrt{\gamma_{1,j}}&\mid &0\leq j<n\\
b_{1,j}^2-\gamma_{1,j}&\mid& j\geq n\\
b_{i,j}^2-\gamma_{i,j}&\mid & i \in [2,r] \end{pmatrix}
\]
of the polynomial ring $\kk[b_{i,j} \mid i \in [1,r], j \in \ZZ]$. Fix a monomial order on the polynomial ring by setting $b_{i,j}<b_{p,q}$ if $i<p$ or if $i=p$ and $j<q$, and using lexicographic order. With this ordering the generators of the ideal $J_n$ form a Gr\"{o}bner basis because their leading terms are relatively prime. The reduction of $b_{1,n}+\sqrt{\gamma_{1,n}}$ by this Gr\"{o}bner basis does not yield zero, showing that $I_{n}\subsetneq I_{n+1}$.
\end{proof}
\begin{theorem}\label{KrullDim}
The Krull dimension of the ring $B$ is one.
\end{theorem}
\begin{proof}
We first claim that $R\subseteq B$ is an integral extension. Indeed, it is clear that for each $i \in [1,r]$ and each $j \in \ZZ$, that $b_{i,j}$ is integral over $R$. Since $B$ is generated by the $b_{i,j}$ over $R$ then $B$ is integral over $R$ by \cite[Corollary 5.3]{AtMcD}. Finally, by \cite[Exercise 9.2]{Mats} the Krull dimension of $B$ is the same as the Krull dimension of $R$, which is one.
\end{proof}

\begin{remark} \label{DirectLimitRemark}
Let $X$ be the set $\{(i,j)\mid i \in [1,r] \text{ and } j\in\mathbb{Z}\}$. Since $X$ is countable, we can fix an enumeration $X=\{x_m\mid m\in\mathbb{N}\}$. If $x_m=(i,j)$ then we write $b_{i,j}$ as $b_{x_m}$ and $\sigma^{-j}((z-\m _i)^{n_i})$ as $g_m$. Denote by $B_m$ the following ring
\[
B_m=\frac{R[b_{x_1},\ldots,b_{x_m}]}{\left(b_{x_t}^2-g_t\mid t \in [1,m] \right)}.
\]
Then 
\[
B_{m+1}=\frac{B_m[b_{x_{m+1}}]}{\left(b_{x_{m+1}}^2-g_{m+1}\right)}
\]
with obvious maps $B_m\rightarrow B_{m+1}$, and we have that
\begin{equation}\label{DirectLimit}
B =\underset{m}{\varinjlim}\;B_m.
\end{equation}
\end{remark}

\begin{definition}
A commutative ring is said to be \emph{coherent} if every finitely generated ideal is finitely presented.
\end{definition}

\begin{proposition}\label{CoherentDomain}
\begin{enumerate}
\item If $n_i$ is odd for all $i \in [1, r]$, then $B$ is a domain.
\item The ring $B$ is coherent.
\end{enumerate}
\end{proposition}
\begin{proof}
(1) Since the direct limit of domains is a domain, by \eqref{DirectLimit} it is enough to show that the rings $B_m$ are all domains. We set $B_{-1}$ to be $R$ (which is a domain) and proceed by induction on $m$. Since
\[
B_{m+1}=\frac{B_m[b_{x_{m+1}}]}{\left(b_{x_{m+1}}^2-g_{m+1}\right)},
\]
and by assumption $\sqrt{g_{m+1}}\not\in B_m$, it is straightforward to see that $b^2_{x_{m_1}}-g_{m+1}$ is a prime element in the domain $B_m[b_{x_{m+1}}]$ and therefore $B_{m+1}$ is a domain.

(2) By \cite[Theorem 2.3.3]{Glaz} it suffices to show that $B_{m}$ is a flat extension of $B_{m-1}$. Indeed $B_m=B_{m-1}\oplus b_{x_m}B_{m-1}$, and hence the extension is free.
\end{proof}

Let $I  \subseteq [1,r]$ consist of those $i$ such that $n_i$ is even. If $I$ is empty, then the proposition above says that $B$ is a domain. In the following lemma, we classify the minimal prime ideals of $B$ when $I$ is not empty.

\begin{lemma}\label{MinimalPrimesLemma}
Suppose $I \neq \emptyset$. An ideal of $B$ is a minimal prime if and only if it is of the form
\begin{equation}
\label{MinimalPrimeIdeal}
\mathfrak{p}=
\left(b_{i,j}-(-1)^{\varepsilon_{i,j}}\sigma^{-j}\left((z-\alpha_i)^{\frac{n_i}{2}}\right) \st i \in I, j\in\mathbb{Z}\right)
\end{equation}
for some choice of $\varepsilon_{i,j}\in\{0,1\}$ for each $i \in I$ and $j\in\mathbb{Z}$.

\end{lemma}
\begin{proof}
Let $S$ be the ring $R[b_{i,j}\mid i\in[1,r], j\in\mathbb{Z}]$ and let $S'$ be the ring $R[b_{i,j}\mid i \in [1,r] \setminus I , j\in\mathbb{Z}]$. The prime ideals of $B$ correspond to prime ideals of $S$ containing the ideal 
\begin{equation}
\left(b_{i,j}^2 - \sigma^{-j}((z - \alpha_i  )^{n_i} )\mid i\in [1,r], j \in \ZZ\right).
\end{equation}
For each $i \in I$ and $j \in \ZZ$, choose $\varepsilon_{i,j} \in \{0,1\}$. Now define the map
\[
\rho:S\rightarrow S'\rightarrow\frac{S'}{(b_{i,j}^2-\sigma^{-j}((z-\alpha_i)^{n_i})\mid n_i\mathrm{\;odd})},
\]
where the first map is the evaluation map that sends $b_{i,j}$ to $(-1)^{\varepsilon_{i,j}}\sigma^{-j}\left((z-\alpha_i)^{\frac{n_i}{2}}\right)$ when $n_i$ is even and the second map is the canonical projection. The kernel of $\rho$ is generated by the generators of the ideal in \eqref{MinimalPrimeIdeal} and by the elements \hbox{$b_{i,j}^2-\sigma^{-j}((z-\alpha_i)^{n_i})$} for $i\in [1,r]\backslash I$. Since $S/\mathrm{Ker}\;\rho$ is isomorphic to $B/\mathfrak{p}$ and since 
\[\frac{S'}{\left(b_{i,j}^2-\sigma^{-j}((z-\alpha_i)^{n_i})\st i \in [1,r]\setminus I, j\in \ZZ\right)}
\]
is a domain by Proposition~\ref{CoherentDomain}, we deduce that the ideal $\pp$ is prime in $B$. To see that $\pp$ is a minimal prime we observe that if a prime ideal contains $b_{i,j}^2-\sigma^{-j}((z-\alpha_i)^{n_i})$ with $n_i$ even then it contains $b_{i,j}-\sigma^{-j}\left((z-\alpha_i)^{\frac{n_i}{2}}\right)$ or $b_{i,j}+\sigma^{-j}\left((z-\alpha_i)^{\frac{n_i}{2}}\right)$.
\end{proof}

The proof of the following result is a straightforward generalization of the proof of \cite[Proposition 4.6]{woncomm}. We reproduce the argument here for completeness.

\begin{corollary}\label{reduced}
The ring $B$ is reduced.
\end{corollary}
\begin{proof}
Let $I \subseteq [1,r]$ consist of those $i \in [1,r]$ such that $n_i$ is even. If $I = \emptyset$ then $B$ is a domain so we may assume that $I \neq \emptyset$. We show that the intersection of all the minimal prime ideals of $B$ is $(0)$. We first set the following notation: let $\ulJ=(J_1,\ldots,J_r)\in\Gamma$ and set $B(\ulJ)$ to be the $\kk$-subalgebra of $B$ generated by $\{z\}\cup\{b_{i,j}\mid i\in[1,r], j\in J_i\}$ if $R=\kk[z]$ or by $\{z,z^{-1}\}\cup\{b_{i,j}\mid i \in [1,r], j\in J_i\}$ if $R=\kk[z,z^{-1}]$. Hence
\[
B(\ulJ)\cong\frac{R[b_{i,j}\mid i \in [1,r], j\in J_i]}{\left(b_{i,j}^2-\sigma^{-j}((z-\alpha_i)^{n_i}) \st i \in[1,r], j\in J_i\right)}.
\]
Let $a$ be an element in the intersection of all minimal prime ideals of $B$. We can write $a$ as a sum of finitely many $\Gamma$-homogeneous terms, so $a$ is in an element of $B(\ulJ)$ for some $\ulJ\in\Gamma$. Fix $i \in I$ and suppose $j\in J_i$. From now on we will work in $B(\ulJ)$. Since
\[
a\in\begin{pmatrix} b_{i,j}+\sigma^{-j}((z-\alpha_i)^{\frac{n_i}{2}})& & \\
b_{i,j'}-\sigma^{-j'}((z-\alpha_i)^{\frac{n_i}{2}})&\mid & j'\in J_i\backslash\{j\}\\
b_{i',j'}-\sigma^{-j'}((z-\alpha_{i'})^{\frac{n_{i'}}{2}})&\mid&i'\in I\backslash\{i\}, j'\in J_{i'}
\end{pmatrix}
\]
and
\[
a\in\begin{pmatrix} b_{i,j}-\sigma^{-j}((z-\alpha_i)^{\frac{n_i}{2}})& & \\
b_{i,j'}-\sigma^{-j'}((z-\alpha_i)^{\frac{n_i}{2}})&\mid & j'\in J_i\backslash\{j\}\\
b_{i',j'}-\sigma^{-j'}((z-\alpha_{i'})^{\frac{n_{i'}}{2}})&\mid&i'\in I\backslash\{i\}, j'\in J_{i'}
\end{pmatrix}
\]
we can write
\begin{equation}\label{4.7Rob}
a=\left(b_{i,j}+\sigma^{-j}\left((z-\alpha_i)^{\frac{n_i}{2}}\right)\right)r+s=\left(b_{i,j}-\sigma^{-j}\left((z-\alpha_i)^{\frac{n_i}{2}}\right)\right)r'+s'
\end{equation}
with $r,r'\in B(\ulJ)$ and 
\[
s,s'\in\begin{pmatrix}b_{i,j'}-\sigma^{-j'}((z-\alpha_i)^{\frac{n_i}{2}})&\mid & j'\in J_i\backslash\{j\}\\
b_{i',j'}-\sigma^{-j'}((z-\alpha_{i'})^{\frac{n_{i'}}{2}})&\mid&i'\in I\backslash\{i\}, j'\in J_{i'}
\end{pmatrix}.
\]
Setting $b_{k,\ell}=\sigma^{-\ell}\left((z-\alpha_{k})^{\frac{n_{k}}{2}}\right)$ for all $k \in I$ and $\ell \in J_{k}$ in \eqref{4.7Rob}, the right hand side becomes identically 0, and since we are now working in a quotient of $B(\ulJ)$ isomorphic to the domain $R$ we can deduce that
\[
r\in\left(b_{k, \ell}-\sigma^{-\ell}\left((z-\alpha_{k})^{\frac{n_{k}}{2}}\right)\st k \in I, \ell \in J_k\right).
\]
So
\[
a\in\begin{pmatrix} b_{i,j}^2-\sigma^{-j}((z-\alpha_i)^{n_i})& & \\
b_{i,j'}-\sigma^{-j'}((z-\alpha_i)^{\frac{n_i}{2}})&\mid & j'\in J_i\backslash\{j\}\\
b_{i',j'}-\sigma^{-j'}((z-\alpha_{i'})^{\frac{n_{i'}}{2}})&\mid&i'\in I\backslash\{i\}, j'\in J_{i'}
\end{pmatrix}.
\]
But $b_{i,j}^2-\sigma^{-j}((z-\alpha_i)^{n_i})=0$ in $B$, hence
\[
a\in\begin{pmatrix} 
b_{i,j'}-\sigma^{-j'}((z-\alpha_i)^{\frac{n_i}{2}})&\mid & j'\in J_i\backslash\{j\}\\
b_{i',j'}-\sigma^{-j'}((z-\alpha_{i'})^{\frac{n_{i'}}{2}})&\mid&i'\in I\backslash\{i\}, j'\in J_{i'}
\end{pmatrix}.
\]
Inducting on the size of $J_i$ and on the size of $I$, we conclude that $a=0$. As a result the intersection of the minimal primes of $B$ is $(0)$.
\end{proof}

\begin{corollary}
The minimal spectrum of $B$, $\mathrm{Min}\;B$, is compact in $\mathrm{Spec}\;B$ with respect to the Zariski topology.
\end{corollary}
\begin{proof}
This follows from Proposition~\ref{CoherentDomain}, Corollary~\ref{reduced}, and \cite[Proposition 1.1]{Matlis}.
\end{proof}

We refer the reader to \cite[Definition 5.1]{HummelMarley} for the definition of the Gorenstein property for coherent rings.

\begin{proposition}
The ring $B$ is a coherent Gorenstein ring.
\end{proposition}
\begin{proof}
By \cite[Proposition 5.11]{HummelMarley}, it suffices to prove that the rings $B_m$ in Remark \ref{DirectLimitRemark} are complete intersections, i.e. if $\mathfrak{q}$ is a maximal ideal of $B_m$ then the localization $(B_m)_\mathfrak{q}$ is a local complete intersection. To prove this we show that the complete intersection defect of $(B_m)_\mathfrak{q}$ is zero, see \cite[Lemma 7.4.1]{IFR}. If $R=\kk[z]$ then $B_m$ is a polynomial ring in $m+1$ variables modulo $m$ relations, if $R=\kk[z,z^{-1}]$ then $B_m$ is isomorphic to a polynomial ring in $m+2$ variables modulo $m+1$ relations, in both cases the complete intersection defect is $\mathrm{dim}\;(B_m)_\mathfrak{q}-1$. As in the proof of Theorem~\ref{KrullDim} the rings $B_m$ have Krull dimension one, and therefore \hbox{$\mathrm{dim}\;(B_m)_\mathfrak{q}\leq 1$}. Since the complete intersection defect is always nonnegative, it follows that it must be zero.
\end{proof}

\bibliography{refs}

\providecommand{\bysame}{\leavevmode\hbox to3em{\hrulefill}\thinspace}
\providecommand{\MR}{\relax\ifhmode\unskip\space\fi MR }
\providecommand{\MRhref}[2]{%
  \href{http://www.ams.org/mathscinet-getitem?mr=#1}{#2}
}
\providecommand{\href}[2]{#2}
\begin{thebibliography}{VdBVO89}

\bibitem[AM69]{AtMcD}
M.~F. Atiyah and I.~G. Macdonald, \emph{Introduction to commutative algebra},
  Addison-Wesley Publishing Co., Reading, Mass.-London-Don Mills, Ont., 1969.
  \MR{0242802}

\bibitem[AS95]{AS}
Michael Artin and J~Toby Stafford, \emph{Noncommutative graded domains with
  quadratic growth}, Inv. Math. \textbf{122} (1995), no.~2, 231--276.

\bibitem[Avr10]{IFR}
Luchezar~L. Avramov, \emph{Infinite free resolutions}, Six lectures on
  commutative algebra, Mod. Birkh\"{a}user Class., Birkh\"{a}user Verlag,
  Basel, 2010, pp.~1--118. \MR{2641236}

\bibitem[AZ94]{AZ}
Michael Artin and James~J Zhang, \emph{Noncommutative projective schemes}, Adv.
  Math. \textbf{109} (1994), no.~2, 228--287.

\bibitem[Bav92]{bavkdim}
Vladimir~V Bavula, \emph{Simple {$D[X,Y;\sigma,a]$}-modules}, Ukrainian Math.
  J. \textbf{44} (1992), no.~12, 1500--1511. \MR{1215036}

\bibitem[Bav93]{bav}
\bysame, \emph{Generalized {W}eyl algebras and their representations}, St.
  Petersburg Math. J. \textbf{4} (1993), no.~1, 71--92.

\bibitem[Bav96]{bavgldim}
\bysame, \emph{Global dimension of generalized {W}eyl algebras}, Representation
  theory of algebras ({C}ocoyoc, 1994), CMS Conf. Proc., vol.~18, Amer. Math.
  Soc., Providence, RI, 1996, pp.~81--107. \MR{1388045}

\bibitem[BR16]{BR}
Jason Bell and Daniel Rogalski, \emph{Z-graded simple rings}, Trans. Amer.
  Math. Soc. \textbf{368} (2016), 4461--4496.

\bibitem[GG82]{GG}
Robert Gordon and Edward~L. Green, \emph{Graded {A}rtin algebras}, J. Algebra
  \textbf{76} (1982), no.~1, 111--137. \MR{659212}

\bibitem[Gla89]{Glaz}
Sarah Glaz, \emph{Commutative coherent rings}, Lecture Notes in Mathematics,
  vol. 1371, Springer-Verlag, Berlin, 1989. \MR{999133}

\bibitem[Haz16]{hazrat}
Roozbeh Hazrat, \emph{Graded rings and graded {G}rothendieck groups}, London
  Mathematical Society Lecture Note Series, Cambridge University Press, 2016.

\bibitem[HM09]{HummelMarley}
Livia Hummel and Thomas Marley, \emph{The {A}uslander-{B}ridger formula and the
  {G}orenstein property for coherent rings}, J. Commut. Algebra \textbf{1}
  (2009), no.~2, 283--314. \MR{2504937}

\bibitem[Hod92]{hodges2}
Timothy~J. Hodges, \emph{Morita equivalence of primitive factors of {$U({\rm
  sl}(2))$}}, Kazhdan-{L}usztig theory and related topics ({C}hicago, {IL},
  1989), Contemp. Math., vol. 139, Amer. Math. Soc., Providence, RI, 1992,
  pp.~175--179. \MR{1197835}

\bibitem[Hod93]{hodges}
Timothy~J Hodges, \emph{Noncommutative deformations of type-{A} {K}leinian
  singularities}, J. Algebra \textbf{161} (1993), no.~2, 271--290.

\bibitem[IN05]{IN}
F.~Casta\~{n}o Iglesias and C.~N\v{a}st\v{a}sescu, \emph{Graded version of
  {K}ato-{M}\"{u}ller theorem. {A}pplications}, Comm. Algebra \textbf{33}
  (2005), no.~6, 1873--1891. \MR{2150848}

\bibitem[Jor92]{jordan2}
David~A. Jordan, \emph{Krull and global dimension of certain iterated skew
  polynomial rings}, Abelian groups and noncommutative rings, Contemp. Math.,
  vol. 130, Amer. Math. Soc., Providence, RI, 1992, pp.~201--213. \MR{1176120}

\bibitem[Kat73]{kato}
Toyonori Kato, \emph{{$U$}-distinguished modules}, J. Algebra \textbf{25}
  (1973), 15--24. \MR{340345}

\bibitem[Mat85]{Matlis}
Eben Matlis, \emph{Commutative semicoherent and semiregular rings}, J. Algebra
  \textbf{95} (1985), no.~2, 343--372. \MR{801272}

\bibitem[Mat89]{Mats}
Hideyuki Matsumura, \emph{Commutative ring theory}, second ed., Cambridge
  Studies in Advanced Mathematics, vol.~8, Cambridge University Press,
  Cambridge, 1989, Translated from the Japanese by M. Reid. \MR{1011461}

\bibitem[MR87]{mcconnell}
J.C. McConnell and J.C. Robson, \emph{Noncommutative {N}oetherian {R}ings},
  American Mathematical Society, 1987.

\bibitem[M{\" u}l74]{muller}
Bruno~J. M{\" u}ller, \emph{The quotient category of a {M}orita context}, J.
  Algebra \textbf{28} (1974), 389--407. \MR{447336}

\bibitem[RS10]{richardsolotar}
Lionel Richard and Andrea Solotar, \emph{On {M}orita equivalence for simple
  {G}eneralized {W}eyl algebras}, Algebr. Represent. Theory \textbf{13} (2010),
  no.~5, 589--605.

\bibitem[Shi10]{shipman}
Ian Shipman, \emph{Generalized {W}eyl algebras: category {$\mathscr{O}$} and
  graded {M}orita equivalence}, J. Algebra \textbf{323} (2010), no.~9,
  2449--2468. \MR{2602389}

\bibitem[Sie09]{sierra}
Susan~J Sierra, \emph{Rings graded equivalent to the {W}eyl algebra}, J.
  Algebra \textbf{321} (2009), no.~2, 495--531.

\bibitem[Sie17]{sierraprivate}
\bysame, {Private communication}, 2017.

\bibitem[Smi00]{smith2}
S~Paul Smith, \emph{Non-commutative algebraic geometry}, June 2000.

\bibitem[Smi11]{smith}
\bysame, \emph{A quotient stack related to the {W}eyl algebra}, J. Algebra
  \textbf{345} (2011), no.~1, 1--48.

\bibitem[Sta82]{St2}
J.~T. Stafford, \emph{Homological properties of the enveloping algebra {$U({\rm
  Sl}_{2})$}}, Math. Proc. Cambridge Philos. Soc. \textbf{91} (1982), no.~1,
  29--37. \MR{633253}

\bibitem[{Sta}19]{stacks-project}
The {Stacks Project authors}, \emph{\itshape {S}tacks {P}roject},
  \url{http://stacks.math.columbia.edu}, 2019.

\bibitem[VdB96]{vdb3}
Michel Van~den Bergh, \emph{A translation principle for the four-dimensional
  {S}klyanin algebras}, J. Algebra \textbf{184} (1996), no.~2, 435--490.
  \MR{1409223}

\bibitem[VdBVO89]{VV}
M.~Van~den Bergh and F.~Van~Oystaeyen, \emph{Lifting maximal orders}, Comm.
  Algebra \textbf{17} (1989), no.~2, 341--349. \MR{978479}

\bibitem[Won18a]{woncomm}
Robert Won, \emph{The noncommutative schemes of generalized {W}eyl algebras},
  J. Algebra \textbf{506} (2018), 322--349.

\bibitem[Won18b]{wonpic}
\bysame, \emph{The {P}icard group of the graded module category of a
  generalized {W}eyl algebra}, J. Algebra \textbf{493} (2018), 89--134.
  \MR{3715206}

\end{thebibliography}
\bibliographystyle{amsalpha}

\end{document}